\newif\ifdraft

\let\ifdraft\iffalse 

\ifdraft
    \documentclass[reqno]{amsart}						
    \usepackage{microtype}
    \usepackage{fullpage}
\else
    \documentclass[smallcondensed]{svjour3}
    \smartqed  

    \makeatletter
    \def\cl@chapter{\@elt {theorem}}
    \makeatother
\fi

\ifdraft
\usepackage{amsmath}						
\usepackage{amssymb}						
\usepackage{amsfonts}						

\usepackage{amsthm}

\usepackage[foot]{amsaddr}						
\usepackage{array}

\usepackage{mathtools}						

\else
\usepackage{amsmath}
\usepackage{amsthm}
\usepackage{amssymb}
\fi

\usepackage[utf8]{inputenc}							
\usepackage[T1]{fontenc}						
\usepackage{lmodern}

\ifdraft
\usepackage{libertine}
\usepackage[libertine,libaltvw,cmintegrals,varbb]{newtxmath}
\fi

\usepackage{dsfont}							

\usepackage[
cal=cm,
]
{mathalfa}

\usepackage[labelfont={bf,small},labelsep=colon,font=small]{caption}	
\usepackage[dvipsnames,svgnames]{xcolor}						
\colorlet{MyBlue}{DodgerBlue!75!Black}
\colorlet{MyGreen}{DarkGreen!85!Black}

\usepackage{subcaption}
\usepackage{tikz}	\usetikzlibrary{patterns}
\usepackage{pgfplots}
\pgfplotsset{compat=1.14}
\usepgfplotslibrary{groupplots}

\usepgfplotslibrary{external}   
\usepackage[many]{tcolorbox}
\usepackage[algo2e]{algorithm2e}
\usepackage{algorithm}
\usepackage{array}
\usepackage{multicol}
\usepackage{multirow}

\newtcolorbox{worker}[1]{%
    tikznode boxed title,
    boxsep = -1mm,
    enhanced,
    arc=0mm,
    interior style={white},
    boxrule=0.5mm,
    width=0.95\linewidth,
    attach boxed title to top center= {yshift=-\tcboxedtitleheight/2},
    fonttitle=\bfseries,
    colbacktitle=white,coltitle=black,
    boxed title style={size=normal,colframe=white,boxrule=0pt},
    title={#1}
    }

\usepackage{acronym}						
\usepackage{booktabs}						
\usepackage{latexsym}						
\usepackage{paralist}						
\usepackage{xspace}							

\ifdraft
\usepackage[sort&compress]{natbib}					

\fi

\usepackage{hyperref}
\hypersetup{
final,
colorlinks=true,
linktocpage=true,
pdfstartview=FitH,
breaklinks=true,
pdfpagemode=UseNone,
pageanchor=true,
pdfpagemode=UseOutlines,
plainpages=false,
bookmarksnumbered,
bookmarksopen=false,
bookmarksopenlevel=1,
hypertexnames=true,
pdfhighlight=/O,
urlcolor=Maroon,linkcolor=MyBlue!60!black,citecolor=DarkGreen!70!black,	
pdftitle={},
pdfauthor={},
pdfsubject={},
pdfkeywords={},
pdfcreator={pdfLaTeX},
pdfproducer={LaTeX with hyperref}
}

\numberwithin{equation}{section}						
\usepackage[sort&compress,capitalize,nameinlink]{cleveref}			

\ifdraft
    \theoremstyle{plain}
    \newtheorem{theorem}{Theorem}						
    \newtheorem{corollary}[theorem]{Corollary}					
    \newtheorem{corollary*}{Corollary}						
    \newtheorem{lemma}[theorem]{Lemma}					

    \theoremstyle{definition}
    \newtheorem{definition}[theorem]{Definition}				
    \newtheorem{definition*}{Definition}					
    \newtheorem{assumption}{Assumption}					
    \newtheorem{assumption*}{Assumptions}					

    \theoremstyle{definition}
    \newtheorem{remark}{Remark}						
    \newtheorem{remark*}{Remark}						
    \newtheorem{example}{Example}						
    \newtheorem{example*}{Example}						
\else
    \newtheorem{assumption}{Assumption}
\fi

\ifdraft
\fi

\numberwithin{equation}{section}						
\numberwithin{theorem}{section}						
\numberwithin{remark}{section}						
\numberwithin{example}{section}						

\usepackage{beramono}
\usepackage{listings}

\lstdefinelanguage{Julia}%
  {morekeywords={abstract,break,case,catch,const,continue,do,else,elseif,%
      end,export,false,for,function,immutable,import,importall,if,in,%
      macro,module,otherwise,quote,return,struct,switch,true,try,type,typealias,%
      using,while},%
   sensitive=true,%
   alsoother={},%
   morecomment=[l]\#,%
   morecomment=[n]{\#=}{=\#},%
   morestring=[s]{"}{"},%
   morestring=[m]{'}{'},%
}[keywords,comments,strings]%

\floatstyle{ruled}
\newfloat{Algorithm}{h!}{lop}

\newcommand{\M}{\mathsf{M}}

\DeclareMathOperator*{\argmin}{argmin}
\DeclareMathOperator{\dist}{dist}
\newcommand{\RR}{\mathbb{R}}

\newcommand{\prox}{\mathbf{prox}}
\newcommand{\Aff}{\operatorname{Aff}}

\newcommand{\idtest}{\mathsf{T}}

\usepackage{xifthen}

\newcommand{\idtestargs}[1][]{%
  \ifthenelse{\isempty{#1}}%
    {\idtest (\{x_{\ell}\}_{\ell\leq k},\{y_{\ell}\}_{\ell < k})}
    {\idtest^{#1} (\{x_{\ell}\}_{\ell\leq k},\{y_{\ell}\}_{\ell < k})}
}

\newcommand{\T}{\mathcal{T}_\gamma}

\usepackage{etoolbox}
\ifdef{\C}{                          
    \renewcommand{\C}{\mathcal{C}}
}{
    \newcommand{\C}{\mathcal{C}}
}

\newcommand{\ri}{\mathrm{ri}~}

\newcommand{\review}[1]{#1}
\newcommand{\secondreview}[1]{#1}
\usepackage[normalem]{ulem}
\newcommand{\reviewst}[1]{}

\newcommand{\rank}{\operatorname{rank}}

\def\addlegendimage{\csname pgfplots@addlegendimage\endcsname}

\begin{document}

\title{On the Interplay between Acceleration and Identification for the Proximal Gradient algorithm}
\author{Gilles Bareilles \and Franck Iutzeler}

\ifdraft
\else
\titlerunning{On the Interplay between Acceleration and Identification}
\institute{G. Bareilles \and F. Iutzeler \at Univ. Grenoble Alpes \\
              Laboratoire Jean Kuntzmann \\
              \email{firstname.lastname@univ-grenoble-alpes.fr}           
}

\date{Received: date / Accepted: date}
\fi

\maketitle

\begin{abstract}
    In this paper, we study the interplay between acceleration and structure identification for the proximal gradient algorithm. \review{While acceleration is generally beneficial in terms of functional decrease, we report and analyze several cases where its interplay with identification has negative effects on the algorithm behavior (iterates oscillation, loss of structure, etc.).} Then, we present a generic method that tames acceleration when structure identification may be at stake; it benefits from a convergence rate that matches the one of the accelerated proximal gradient under some qualifying condition. We show empirically that the proposed method is much more stable in terms of subspace identification compared to the accelerated proximal gradient method while keeping a similar functional decrease.
\keywords{Proximal Gradient  \and  Accelerated Methods \and Identification}
\end{abstract}

\section{Introduction}

In this paper, we consider composite optimization problems of the form
\begin{equation}
\label{eq:pb}
\min_{x\in\mathbb{R}^n} F(x) :=  f(x) + g(x)
\end{equation}
where $f$ and $g$ are convex functions. We are interested in the case where $f$ is smooth (i.e. differentiable with a Lipschitz-continuous gradient) while $g$ is nonsmooth and enforces some structure to the solutions of \eqref{eq:pb} (e.g. belonging to some set, sparsity, etc.).  This type of problem often appears in signal processing and machine learning in which case the goal is usually to find a point with a low error in the task at hand while maintaining some structure. For instance, taking $g$ as the $\ell_1$ norm promotes sparsity, it is commonly used in the recovery of sparse signals or compressed sensing \cite{candes2006stable}; we refer the reader to \cite{2015-vaiter-ps-review} for an overview.

Problem~\eqref{eq:pb} is typically solved by the proximal gradient algorithm, a special case of Forward-Backward splitting, which alternates a gradient step on the smooth function $f$ and a proximal operator on the nonsmooth function $g$; see e.g. \cite[Chap.~27]{bauschke2011convex}. Moreover, in order to improve the convergence of this algorithm, \emph{accelerated} (also called \emph{fast} or \emph{inertial}) versions have been widely promoted, notably thanks to the popularity of FISTA \cite{beck2009fista}. These modifications consist in producing the next iterate by linearly combining the previous outputs of the proximal operator. Using combination coefficients as per Nesterov's fast method \cite{nesterov1983method} (or similar ones \cite{attouch2016rate,chambolle2015convergence}) improves practically and theoretically the rate of convergence of the proximal gradient (nonetheless to the price of a more involved analysis, see e.g.  \cite{alvarez2001inertial,chambolle2015convergence,liang2019improving}).

Additionally, the handling of the nonsmooth function $g$ by a proximity operator enforces some \emph{structure} on the iterates. For instance, when $g$ is the $\ell_1$-norm, the associated proximity operator, often called soft-thresholding, puts coordinates with small values to zeros, thus producing sparse iterates \cite{donoho1995noising}. More generally, an important property of the proximity operator is that it has the same output for a neighborhood of inputs around the points where $g$ is non-differentiable, which means that reaching a neighborhood of the solution may be enough to capture the optimal structure\footnote{For instance, if $g$ is the absolute value on $\mathbb{R}$, the proximity operator of $g$ is equal to $0$ for all $x\in[-1,1]$ (and $x-\mathrm{sign}(x)$ elsewhere). Thus, if some sequence $x_k$ converges to a point within $(-1,1)$, then the output of the proximity operator of $g$ at $x_k$ will be equal to $0$ for all $k$ greater than some finite, but unknown, $K$. }.
This property is often called \emph{identification} and signals that the iterates generated by proximal methods can have the same structure (sparsity pattern, belonging to a set, rank, etc.) as the final optimum in finite time.
The analysis of this property, including conditions for identification, has attracted a lot of attention in the case of the projected gradient algorithm and then more generally in nonsmooth optimization; see e.g.  \cite{bertsekas1976goldstein,burke1988identification,lewis2002active,drusvyatskiy2013optimality}.
This interest is notably driven by the fact that i) identification helps reducing the dimension of the problem and thus allows a faster computation at each iteration; and ii) the uncovered structure often bears valuable information. For instance, in $\ell_1$-regularized regression problems such as the lasso, the proximal gradient algorithm can identify the non-zero coordinates of the solution and thus the most significant features \cite{tibshirani1996regression}.

\review{
Although the accelerated versions of the proximal gradient are known to be faster globally and locally around the optimum \cite{liang2017localcvFBtype}, there is a transient phase at the moment where the final manifold start being attained by the iterates where the interplay between identification and acceleration can be negative, causing the iterates to leave several times the optimal manifold before identifying it.}
\reviewst{Unfortunately, acceleration and identification may clash in practice.}Indeed, the accelerated proximal gradient is known to suffer from an oscillatory, non-monotonical behavior (see e.g.~\cite[Sec.~5.4]{liang2017localcvFBtype}) which may make the iterates leave the identification neighborhood. Several works of the literature considered modifications of the accelerated proximal gradient aiming at limiting its downsides using heuristic restarts \cite{scheinberg2014fast,o2015adaptive,ito2017unified,catalina2018revisiting} or adaptive acceleration \cite{giselsson2014monotonicity,poon2019trajectory}; unfortunately, most of these results are empirical and lack a refined analysis. A notable exception is when the acceleration is limited to the iterations where the functional value decreases \cite{beck2009fast,li2015accelerated}, in which case the usual convergence results hold.

In this paper, we provide structure-enhancing accelerations of the proximal point algorithm \review{that are aimed at producing iterates with stable identification properties while maintaining a satisfying convergence rate}. First, we motivate this study by illustrating \reviewst{the above-mentioned clash that can happen between acceleration and identification}\review{on several examples how acceleration and identification may interfere destructively, but also at times constructively}. Then, we introduce a proximal gradient method where, at each iteration, a test is performed to decide if acceleration should be performed or not to produce the next iterate \review{based on the structure uncovered by the method}. \review{This test is designed to benefit from the positive above-mentioned interactions while avoiding negative ones}. We provide two simple but efficient tests for which we show that i) both associated algorithms have the same theoretical rate as the accelerated proximal gradient under some qualifying condition; but ii) they exhibit a much more stable identification behavior in practice.

The paper is organized as follows. In \cref{sec:settingsrecalls}, we recall the main properties of the (accelerated) proximal gradient method, present the notion of identification, and illustrate the interplay between acceleration and identification. In \cref{sec:identifprom}, we introduce our structure-enhancing acceleration methods for which we carefully analyze the convergence. In \cref{sec:numexps}, we illustrate the merits of the proposed methods compared to the usual (accelerated) proximal gradient.

\section{Proximal Gradient, Acceleration, and Identification} \label{sec:settingsrecalls}

\subsection{Proximal Gradient}

We make the following usual assumption on the considered composite convex optimization problem \eqref{eq:pb}.

\begin{assumption}[On Problem~\eqref{eq:pb}] \label{assumption1} ~~
\begin{itemize}
    \item[(i)] $f$ is a differentiable convex function with an $L$-Lipschitz gradient;
    \item[(ii)] $g$ is a proper, convex, lower semi-continuous function;
    \item[(iii)] $\argmin_{x\in\mathbb R^n} F(x) \neq \emptyset$ and we define $F^\star = \min_{x\in\mathbb R^n} F(x)$.
\end{itemize}
\end{assumption}

Under these assumptions, the proximal gradient is arguably the baseline method for solving~\eqref{eq:pb}. Its iterations involve a gradient step on the smooth function $f$ and a proximal step on the possibly nonsmooth part $g$:
\begin{equation*} \tag{PG} \label{alg:PG}
    \left|
    \begin{array}{l}
         x_{k+1} = \prox_{\gamma g}(x_k - \gamma \nabla f(x_k))
    \end{array}
    \right.
\end{equation*}
where $\gamma>0$ is a fixed stepsize and the proximity operator of $g$ (see \cite[Chap.~6]{beck2017first} for an overview) is defined by
\begin{equation*}
    \prox_{\gamma g}(u) := \argmin_{w\in\mathbb R^n}\left\{ g(w) + \frac{1}{2\gamma}\|w-u\|^2 \right\} .
\end{equation*}

To lighten the notation, we will denote by $\mathcal T_\gamma$ the proximal gradient operator with step size $\gamma$:
$$
\mathcal T_\gamma(x) = \prox_{\gamma g}(x - \gamma \nabla f(x))
$$
for any $x\in\mathbb{R}^n$; the proximal gradient algorithm then writes $ x_{k+1} = \mathcal T_\gamma (x_k)$.

For any $0 < \gamma < 2/L$, the sequences of functional values $(F(x_k))$ and iterates $(x_k)$ produced by the proximal gradient algorithm converge monotonically and Féjèr monotonically\footnote{That is for all iterates $k>0$, $\|x_{k+1}-x^\star\| \le \|x_k-x^\star\|$ for any minimizer $x^\star\in\argmin_x F(x)$.}  respectively, at rate $1/k$ (more precisely, $F(x_k)-F^\star = \mathcal{O}(1/k)$ and $\|x_k - x_{k-1}\|^2 = \mathcal{O}(1/{k})$); see e.g.~\cite[Chap.~10]{beck2017first}.

\subsection{Acceleration}

The principle of adding an inertial (or momentum) step to accelerate the convergence of optimization methods stems from Nesterov's fast gradient \cite{nesterov1983method} and Polyak's heavy ball \cite{polyak1964some} methods for vanilla gradient descent. It was later generalized and analyzed for the proximal gradient algorithm in \cite{beck2009fista}, from which the accelerated proximal gradient is often called FISTA.

Mathematically, this acceleration consists in adding to the output of the proximal gradient $x_{k+1}$ an inertial term made of the difference between the two previous points $x_{k+1}-x_k$, weighted by an inertial coefficient $\alpha_{k+1}>0$. One iteration therefore reads:
\begin{equation*} \tag{Accel. PG} \label{alg:AccelPG}
    \left|
    \begin{array}{l}
        x_{k+1} = \prox_{\gamma g}(y_k - \gamma \nabla f(y_k)) \\
        y_{k+1} = x_{k+1}+\alpha_{k+1}(x_{k+1}-x_{k})
    \end{array}
    \right.
\end{equation*}
where the inertial sequence $(\alpha_{k})$ is chosen carefully as follows (see e.g.~\cite[Rem.~10.35]{beck2017first}).

\begin{assumption}[On the inertial sequence]
\label{hyp:alpha}
For any $k>0$,  $\alpha_{k+1}=\frac{t_{k}-1}{t_{k+1}}$ with:
\begin{itemize}
\item[(i)] $t_{k+1}^2-t_{k+1} \le t_{k}^2$;
\item[(ii)] $t_k \ge C k  $ for some $C>0$ and $t_0=1$.
\end{itemize}
\end{assumption}

The inertial sequence originally used in \cite{nesterov1983method,beck2009fista} comes from taking $t_k = ( 1+\sqrt{1+4t_{k-1}^{2}} ) / 2$ and $t_1=1$. For that choice and $0<\gamma< 1/L$, functional convergence occurs with a faster $\mathcal O(1/k^2)$ rate.

Other popular choices of the literature include: (i) $t_{k}=(k+a-1)/a$ with $a>2$ which allows to prove the convergence of the iterates in \cite{chambolle2015convergence} and a $o(1/k^2)$ improved functional convergence rate in \cite{attouch2016rate}; (ii) $t_k = (p +\sqrt{q+4t_{k-1}^{2}} ) / 2$, $t_1=1$, $p\in(0,1]$, $q>0$ in the recent \cite{liang2019improving}.

\subsection{Identification}

Let us consider a generic algorithm including a proximal step
\begin{align}
\label{eq:proxident}
        x_k = \prox_{\gamma g}(u_k)
\end{align}
such that $u_k$ converges to a point denoted by $u^\star$ (which implies that $x_k \to x^\star = \prox_{\gamma g}(u^\star)$ by non-expansivity of the proximity operator \cite[Prop.~12.27]{bauschke2011convex}).

We are particularly interested in the case where the limit point $x^\star$ belongs to some manifold of interest $\M$. In that context, identification means that the iterates $(x_k)$ will reach the manifold $\M$ in finite time. \review{For instance, when $g=\|\cdot\|_1$ (for instance in the lasso problem \cite{tibshirani1996regression}), the manifold of interest $\M$ can be equal to the sparsity pattern of $x^\star$ ($\M = \{ x : x_{[i]} = 0 \text{ for all } i \text{ s.t. } x^\star_{[i]} = 0 \}$ where $x_{[i]}$ stands for the $i$-th coordinate of $x$); meaning that identification brings information about the nullity of the coordinates of the problem solution, enabling feature selection or dimension reduction.

For identification to happen, a sufficient condition is that for any $u$ close to $u^\star$, the proximity operator maps $u$ to a point in $\M$. Mathematically, this qualifying condition writes
\begin{align*}
    \tag{QC}
    \label{eq:QC}
   \exists~\varepsilon>0 \text{ such that for all } u\in\mathcal B(u^\star, \varepsilon), \; \prox_{\gamma g}(u) \in\M.
\end{align*}

Considering the iteration \eqref{eq:proxident}, since $u_k\to u^\star$, after some finite but unknown time, $u_k\in\mathcal B(u^\star, \varepsilon)$. The qualifying condition \eqref{eq:QC} then implies that for such iterates $u_k$, $x_k=\prox_{\gamma g}(u_k)$ belongs to $\M$. Therefore after some finite time, $x_k\in\M$. This simple identification result can be formalized as follows.
}

\begin{lemma}[Identification] \label{lemma:identif}
Let $(x_k)$ and $(u_k)$ be a pair of sequences such that $x_k = \prox_{\gamma g}(u^k) \to x^\star = \prox_{\gamma g}(u^\star)$ and $\M$ be a manifold. If $x^\star \in\M$ and \eqref{eq:QC} holds, then, after some finite time, $x_k \in\M $.
\end{lemma}

\review{
This result indicates that the iterates of any converging algorithm including a proximal step will identify the manifold $\M$ in finite time if \eqref{eq:QC} is satisfied. Note that this result depends only on the convergence of the iterates and the optimal pair $(x^\star,u^\star)$ but is independent from the algorithm itself or the initialization point. In particular, it applies to the two algorithms we are concerned with: the Proximal Gradient algorithm and its accelerated version.
}

\review{
\begin{remark}[About the Qualifying constraint] The condition \eqref{eq:QC} is rather general; notably, it does not constrain the manifold nor the algorithm. Nevertheless, we can show that it naturally encompasses usual qualifying constraints of the literature. For instance, if $g$ is partly smooth \cite{hare2004identifying} relative to $\M$ and the non degeneracy condition $(u^\star - x^\star)/\gamma \in \ri \partial g(x^\star)$ is verified, then \eqref{eq:QC} holds. We refer the reader to Appendix~\ref{appendix:QC} for a formal proof.

In the particular case of the proximal gradient method \eqref{alg:PG} or its accelerated version \eqref{alg:AccelPG}, $u^\star = x^\star - \gamma \nabla f(x^\star)$; then, the non degeneracy condition simplifies to $-\nabla f(x^\star) \in \ri ~\partial g(x^\star)$. This non-degeneracy condition, along with partial smoothness, was used in \cite[Th.~3.4]{liang2017localcvFBtype} to show an identification result similar to Lemma~\ref{lemma:identif} for the proximal gradient algorithm and its accelerated variants.
\end{remark}
}

\begin{remark}[Failure cases] \label{rk:QCfailure_example} If one cannot find a ball centered on $u^\star$ mapped to $\M$ by $\prox_{\gamma g}$, \review{then identification may or may not occur, depending on the algorithm and the initialization point. O}ne can have $x^\star\in\M$ but $x_k\not\in\M$ for every iteration \review{of proximal gradient, for example} with $f=\frac{1}{2}(\cdot-1)^2$ and $g=|\cdot|$\review{. The associated} composite problem $\min_{x\in\mathbb R} \frac{1}{2}(x-1)^2 + |x|$ has solution $x^\star=0$ (as the optimality condition writes $0 \in x^\star-1+\partial |x^\star|$). If one defines proximal gradient iterates as  $u_k = x_{k-1}-\gamma\nabla f(x_{k-1})$ and \eqref{eq:proxident} with $\gamma \in (0,1)$, then, for a positive starting point $x_0>0$,  $x_k = (1-\gamma)^k x_0 \to x^\star = 0$. Therefore, when considering the manifold $\M = \{0\}$,  $x^\star$ belongs to $\M$ and yet none of the iterates do. This example illustrates the failure of \eqref{eq:QC}. Indeed, $u_k = \gamma + (1-\gamma)^k x_0 \to u^\star = \gamma$ but there is no ball around $u^\star$ that is mapped to $\M = \{0\}$ since $\prox_{\gamma g}(u^\star + \varepsilon) = \varepsilon \notin \M$ for any $\varepsilon>0$. \review{Note finally that choosing $x_0<0$, or $\gamma = 1$, or adding an inertial step, would make the algorithm identify $\M$ in finite time, illustrating the dependency on the algorithm and the initialization point for such unqualified problems. }
\end{remark}

\subsection{On the practical use of identification}

\review{
Identification results can be somehow nebulous since they only ensure that the iterates will \emph{eventually} reach \emph{some} structure. Nevertheless, identification can still be leveraged in practice in many situations.

More precisely, given an optimization problem, the user is often able to define a collection of subspaces of interest $\C=\{\M_1, \ldots, \M_N\}$ representing the structure that the solution can have. For instance\footnote{Other examples include problems regularized with nuclear norm, for which the target structure is the rank of the current (matrix) iterate $x\in\RR^{m_1\times m_2}$. Manifolds of interest are defined as $\M_i =\{x\in\RR^{m_1\times m_2} : \rank(x)=i\}$, and the collection is $\M=\{\M_0, \ldots, \M_{\min(m_1,m_2)}\}$.
}, in $\ell_1$ regularized problems, the solution is expected to be sparse and thus we will naturally look at collection $\C = \{\M_1,  \ldots, \M_n\}$ where $\M_i = \{x\in\RR^n : x_{[i]}=0\}$. Then, we say that an algorithm {identifies} the optimal structure if, after some finite time, its iterates belong exactly to the same subspaces as $x^\star$; i.e. they have the same sparsity pattern as $x^\star$. As mentioned above, this phenomenon is highly valuable in many applications such as compressed sensing and machine learning for which finding some (near-)optimal pattern is also an important part of the problem.
This multiple identification result can be stated as follows.

\begin{lemma}[Multiple Identification] \label{lemma:identif2}
    Let $(x_k)$ and $(u_k)$ be a pair of sequences such that $x_k = \prox_{\gamma g}(u^k) \to x^\star = \prox_{\gamma g}(u^\star)$ and let $\C = \{\M_1,  \ldots, \M_N\}$ be a collection of manifolds. For any $\M\in\C$:\\
    (i) if $x^\star \in\M$ and \eqref{eq:QC} holds, then, after some finite time, $x_k \in\M $;\\
    (ii) if $x^\star \notin \M$, then, after some finite time, $x_k \notin\M $.
\end{lemma}

\begin{proof}
        The proof of the first point comes directly from Lemma~\ref{lemma:identif}.
        For the second part, consider the union of the manifolds to which $x^\star$ does not belongs $\overline{\M} = \cup \{ \M \in \C : x^\star \notin \M \}$. Then,  $x^\star\in \mathbb{R}^n \setminus \overline{\M}$ which is an open set. This means that as $x_k \to x^\star$, it will also belong to $ \mathbb{R}^n \setminus \overline{\M}$ in finite time and thus will not belong to any manifold in $ \overline{\M}$.
\end{proof}

From this result, we additionally get that $x_k$ will not identify manifolds to which $x^\star$ does not belong provided that the number of considered manifold is finite (which is a rather mild assumption). In addition, we see that this result, as many results in identification, also holds when the manifolds are replaced by mere closed sets (closedness is however fundamental as per the proof of Lemma~\ref{lemma:identif2}). We choose in this paper to stand with the notion of manifold out of consistency with a large part of the literature on identification and since many usual regularizers in signal processing and machine learning enforce a manifold-based partition of the space; see e.g. \cite{fadili2018sensitivity} and references therein.


}

\subsection{Interplay between Acceleration and Identification}
\label{sec:interplay}

\begin{figure}
     \centering
     \begin{subfigure}[t]{0.49\textwidth}
         \centering
        \includegraphics[width=\textwidth]{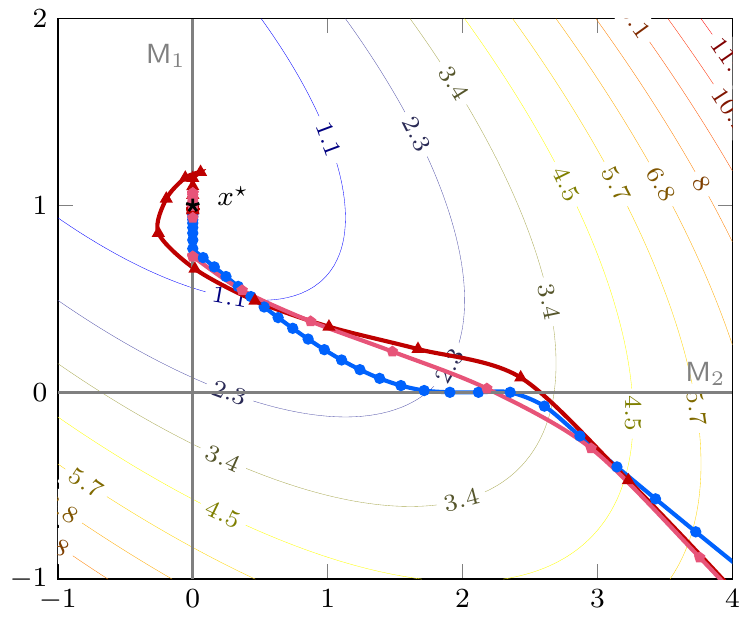}
         \caption{$g(x) = \|x\|_1$ ~~~~ $\M_1 = y$-axis \\ $\M_2 = x$-axis (two linear manifolds) \label{fig:IstaFistaLASSO}}
     \end{subfigure}\hfill
     \begin{subfigure}[t]{0.49\textwidth}
        \centering
        \includegraphics[width=\textwidth]{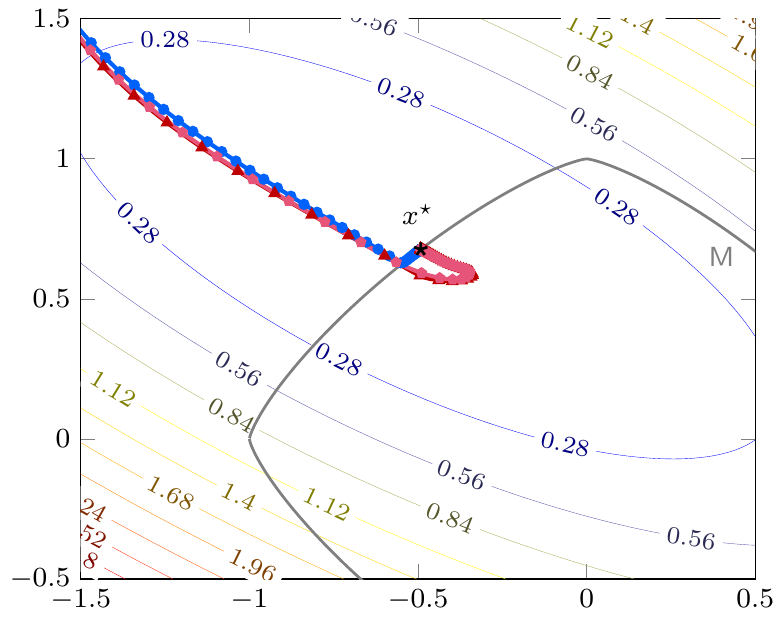}
         \caption{ $g(x) =  \max(0, \|x\|_{1.3}-1)$ \\ $\M = \mathcal S_{\|\cdot\|_{1.3}}(0, 1)$ (one curved manifold) \label{fig:1.3ball}}
    \end{subfigure}\\
    \begin{subfigure}[t]{0.49\textwidth}
        \centering
        \includegraphics[width=\textwidth]{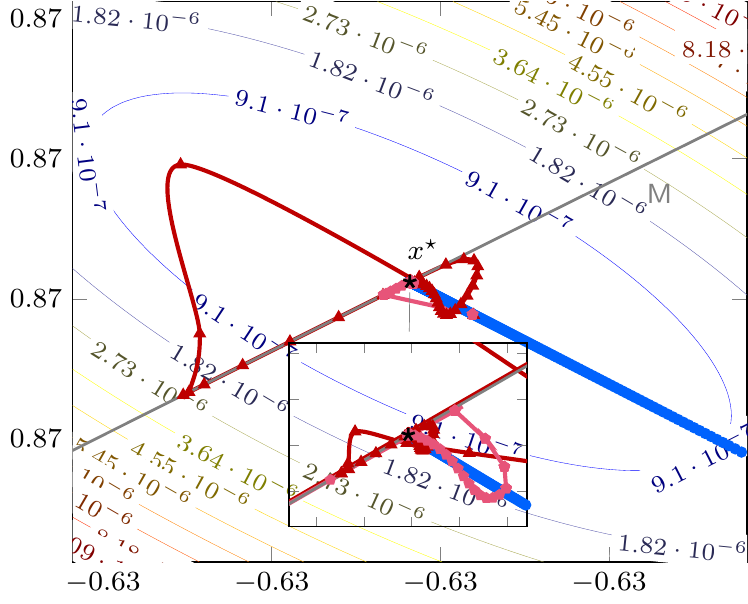}
        \caption{$g(x) =  \max(0, \|x\|_{2.6}-1)$  \\ $\M = \mathcal S_{\|\cdot\|_{2.6}}(0, 1)$ (one curved manifold)\label{fig:2.6ball}}
    \end{subfigure} \hfill
    \begin{subfigure}[t]{0.49\textwidth}
        \centering
        \vspace{-3cm}
        \includegraphics[width=0.7\textwidth]{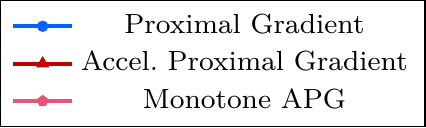}
    \end{subfigure}
    \caption{Iterates behavior for the Proximal Gradient (with and without acceleration)  when minimizing $\|Ax-b\|^2 + g(x)$ for different nonsmooth functions $g$. The candidate manifolds are precised in the captions and represented in gray. \label{fig:identif}}
\end{figure}


In this section, we argue that acceleration \reviewst{may disturb}\review{interferes with} the identification process that occurs with the proximal gradient algorithm\review{, sometimes delaying it, sometimes helping it\footnote{{Note that we are only interested here in the transient phase of identification when the iterates are in the process of reaching the optimal manifold. The behavior of the accelerated proximal gradient is usually better than the vanilla version before this phase and after identification; see e.g. the recent analyses of \cite{liang2017localcvFBtype,apidopoulos2020convergence} and references therein.}}}. We illustrate \review{and detail} this claim by considering the iterates of\reviewst{ the (accelerated)} proximal gradient\review{, accelerated proximal gradient and monotone accelerated proximal gradient, a slightly modified and widely used version of accelerated which ensures a monotonic decrease of iterates sometimes called MFISTA \cite{beck2009fast}. These algorithms are tested} on three problems in $\RR^2$ of the form $\|Ax-b\|^2 + g(x)$ for different nonsmooth functions $g$: the $\ell_1$ norm and the distances to the unit ball in 1.3 and 2.6 norms\footnote{
    $g = \min (\|\cdot\|_{p}-1, 0)$ for $p$ in $\{1.3, 2.6\}$, the proximity operator of which is $\prox_{\gamma g}(u)$ is $u (1-\gamma/ \|u\|_p)$ if $\|u\|_p > 1+\gamma$, $u/\|u\|_p$ if $1 \le \|u\|_p \le 1+\gamma$, and $u$ otherwise.
}. The natural manifolds of interest for these nonsmooth functions are respectively the set of cartesian axes of $\mathbb R^n$ and the unit spheres of the 1.3 and 2.6 norm, which we denote by $\mathcal S_{\|\cdot\|_{1.3}}(0,1)$ and $\mathcal S_{\|\cdot\|_{2.6}}(0,1)$. Figure~\ref{fig:identif} highlights three interesting behaviors:
\begin{itemize}
    \item[(1)] Upon reaching a manifold, the inertial term of the accelerated version\review{s} will not be aligned with the manifold in general and thus will have a non null orthogonal component to $\M$. Unless that orthogonal component is small enough, it will cause iterates to miss the manifold and go beyond it. \cref{fig:IstaFistaLASSO} illustrates this point for \review{accelerated proximal gradient over} linear manifolds: the\reviewst{ accelerated} iterates go past the optimal manifold $\M_1$ twice before reaching it definitively, while the proximal gradient iterates identify it directly. \review{\cref{fig:2.6ball} shows the same overshooting behavior of accelerated proximal gradient.} In \cref{fig:1.3ball}\reviewst{ and \ref{fig:2.6ball}}, while proximal gradient iterates identify the optimal manifold definitively,\reviewst{ accelerated} iterates \review{of both accelerated versions} go beyond it, only to reach it again after several iterations.
    \item[(2)] In the case of a curved (non-affine) manifold $\M$, the interplay between the curvature and the inertial term can cause the iterates to leave the manifold. In \cref{fig:2.6ball},\reviewst{ accelerated} iterates \review{of both accelerated versions} reach $\M$ a first time but leave it after some iterations. It turns out that these iterates do reach $\M$ again, but only to leave it again some time later and have this phenomenon happen periodically (this periodicity can be seen \review{in the zoom of \cref{fig:2.6ball} and numerically for both problems with curved manifolds}\reviewst{numerically in \ref{fig:numexps_2.6ball_subopt}}).
    \item[(3)] Acceleration also has some kind of exploratory behavior that increases the chances to encounter an optimal manifold, which can be helpful with problems not verifying \eqref{eq:QC}. In \cref{fig:2.6ball} where \eqref{eq:QC} does not hold, the\reviewst{ accelerated} iterates \review{of both accelerated versions} reach the optimal manifold, at least for some time, while proximal gradient iterates never does \review{(this phenomenon is also illustrated in \cref{rk:QCfailure_example})}.
\end{itemize}

\section{Identification-promoting Acceleration} \label{sec:identifprom}

Based on these remarks, we introduce in this section a proximal gradient method with \emph{provisional} acceleration. By provisional, we mean that acceleration is carried out as long as it does not jeopardize a fast identification. The general goal of this type of algorithm is to maintain a satisfying practical and theoretical rate \review{by accelerating as soon as possible} while making the iterates stick to the identified structure as much as possible.

We first start by laying out a generic provisionally accelerated proximal gradient and then propose efficient practical tests to determine whether or not to accelerate an iteration.

\subsection{Generic Provisionally Accelerated Proximal Gradient algorithm} \label{subsec:condalg}

Let us denote by $\idtest$ the boolean-valued function that will determine if an iteration should be accelerated or not. It expects as an input some previously computed iterates and returns $1$ if acceleration should be performed, $0$ otherwise. The proposed proximal gradient with provisional acceleration then writes:
\begin{equation*}
    \tag{Prov. Alg.}\label{eq:condalg}
    \left|
    \begin{array}{l}
        y_{k}= \left\{
        \begin{array}{ll}
            x_k + \alpha_k (x_k - x_{k-1}) & \textrm{ if } \idtestargs = 1 \\
            x_k  & \textrm{ otherwise}
        \end{array} \right. \\
        {u_{k+1}=y_{k}-\gamma\nabla f(y_k) } \\
        {x_{k+1}=\prox_{\gamma g}(u_{k+1})} \\
    \end{array}
\right.
\end{equation*}

A general bound on suboptimality can be derived for this algorithm, independently of the value of the test $\idtest$, as stated in the following lemma.

\begin{lemma}\label{thm:general_func_cv}
Let Assumptions \ref{assumption1} and \ref{hyp:alpha} hold and take $\gamma>0$. Then, the iterates of \eqref{eq:condalg} verify
\begin{align}
    \label{eq:base}
    t_n^2 \left[F(x_{n+1})-F^\star\right] \le &- \sum_{k=0}^n \frac{1-\gamma L}{2 \gamma} t_k^2 \|x_{k+1}-y_k\|^2 + \frac{1}{2 \gamma}\|x_0-x^\star\|^2 \\
  \nonumber  & ~~~~ + \sum_{k=1}^n (1-\idtestargs) \frac{1}{2\gamma}\|x_k-x^\star\|^2
\end{align}
\end{lemma}

\begin{proof} We start from the standard accelerated descent inequality (recalled in Lemma~\ref{lem:accdescent} of Appendix~\ref{appendix:functionalcv}), with $v_k := F(x_{k+1})-F^\star$:
    \begin{align} \label{eq:funcCV_genericbound}
        t_{k}^2 v_k &- t_{k-1}^2 v_{k-1} \le  -\frac{1-\gamma L}{2 \gamma} \|t_k x_{k+1}- t_k y_k\|^2  \\
      \nonumber  & ~~~ - \frac{1}{2\gamma}\|t_k x_{k+1} - (t_k-1)x_k -x^\star\|^2 + \frac{1}{2\gamma} \|t_k y_k - (t_k-1)x_k -x^\star\|^2.
    \end{align}

    This equation can be specified to the different extrapolation updates:
    \begin{itemize}
        \item The accelerated update $\idtestargs = 1$ specifies \eqref{eq:funcCV_genericbound} to:
    \begin{align} \label{eq:char_accel}
        t_{k}^2 v_k &- t_{k-1}^2 v_{k-1} \le -\frac{1-\gamma L}{2 \gamma} \|t_k x_{k+1}- t_k y_k\|^2 \\
      \nonumber  & ~~~ - \frac{1}{2\gamma}\|t_k x_{k+1} - (t_k-1)x_k -x^\star\|^2 + \frac{1}{2\gamma} \|t_{k-1} x_k - (t_{k-1}-1)x_{k-1} -y^\star\|^2
    \end{align}
    since the update gives $t_k(y_k-x_k) = (t_{k-1}-1) (x_k-x_{k-1})$.
    \item The proximal gradient update $\idtestargs = 0$ specifies \eqref{eq:funcCV_genericbound} to:
    \begin{align} \label{eq:char_notaccel}
        t_{k}^2 v_k &- t_{k-1}^2 v_{k-1}\le   -\frac{1-\gamma L}{2 \gamma} \|t_k x_{k+1}- t_k y_k\|^2  \\
        \nonumber  & ~~~  - \frac{1}{2\gamma}\|t_k x_{k+1} - (t_k-1)x_k -x^\star\|^2 + \frac{1}{2\gamma} \|x_k -x^\star\|^2
    \end{align}
    since the update is $y_k=x_k$.
    \end{itemize}

    Both \cref{eq:char_accel,eq:char_notaccel} can be summarized, at the cost of introducing some error when acceleration is performed, as:
    \begin{equation*}
        \begin{aligned}
            t_{k}^2 v_k &- t_{k-1}^2 v_{k-1} \le -\frac{1-\gamma L}{2 \gamma} \|t_k x_{k+1}- t_k y_k\|^2 \\
            &- \frac{1}{2\gamma}\|t_k x_{k+1} - (t_k-1)x_k -x^\star\|^2 + \frac{1}{2\gamma} \|t_{k-1} x_k - (t_{k-1}-1)x_{k-1} -x^\star\|^2\\
            &+ (1-\idtestargs) \frac{1}{2\gamma}\|x_k-x^\star\|^2
        \end{aligned}
    \end{equation*}

    A functional error bound can now be deduced, by summing these inequalities up to iteration $n$ and re-arranging terms:
    \begin{equation} \label{eq:subopt_globalsum}
        \begin{aligned}
            t_n^2 [F(x_{n+1}) &-F^\star] \le F(x_1)-F^\star  - \sum_{k=1}^n \frac{1-\gamma L}{2 \gamma} \|t_k x_{k+1}- t_k y_k\|^2 \\
            &- \frac{1}{2\gamma}\|t_n x_{n+1} - (t_n-1)x_n -x^\star\|^2 + \frac{1}{2\gamma} \|x_1-x^\star\|^2\\
            &+ \sum_{k=1}^n (1-\idtestargs) \frac{1}{2\gamma}\|x_k-x^\star\|^2
        \end{aligned}
    \end{equation}
    where $t_0=1$.
    Suboptimality at first iteration can be approximated by applying the descent lemma (Lemma~\ref{lem:descent}) to $(x=x^\star, y=y_0)$:
    \begin{align*}
        F(x_1)-F^\star &\le -\frac{2-\gamma L}{2 \gamma}\|x_1-y_0\|^2 + \frac{1}{\gamma} \langle y_0-x^\star, x_1-y_0 \rangle \\
        &= -\frac{1-\gamma L}{2 \gamma}\|x_1-y_0\|^2 + \frac{1}{2\gamma}\|y_0-x^\star\|^2 - \frac{1}{2\gamma}\| x_1-x^\star\|^2
    \end{align*}

    Finally, recalling that $y_0=x_0$, applying the previous majoration and $- \frac{1}{2\gamma}\|t_n x_{n+1} - (t_n-1)x_n -x^\star\|^2<0$ to \cref{eq:subopt_globalsum} yields the result.
\end{proof}

\review{At this point, we mention that most articles of the literature concerning adaptive acceleration of the proximal gradient are either heuristic \cite{o2015adaptive,ito2017unified,iutzeler2019generic} or include functional monotonicity tests \cite{beck2009fast,li2015accelerated,giselsson2014monotonicity} (a notable exception being alternated inertia \cite{Iutzeler2018} in a slightly different context).}

\review{Since the originality of our approach is to study the joint effect of acceleration and identification, we need to link the functional convergence coming from the analysis of accelerated methods (see Theorem~\ref{thm:general_func_cv}) with the iterates convergence necessary for identification (see Lemma~\ref{lemma:identif}). This can be done rather generally by considering the following error bound. }

\begin{assumption} \label{EB}
There exists $\beta>0$ and $p\ge 1$ such that
$
    F(x) - F^\star \ge \beta ~\mathrm{dist}^p ~ (x, \argmin F)
$
for all $x$.
\end{assumption}

\review{This type of (global Hölderian) error bound is verified for a large class of functions (for instance proper, convex, lower-semicontinuous, and semi-algebraic functions; see \cite{Bolte2007KL}) and no knowledge about the actual values of $\beta$ and $p$ is needed.
Furthermore, it was shown in \cite{bolte2017error} that for convex functions, this error bound implies that $ \dist(0,\partial F(\mathsf{T}_\gamma(x) ))  $ upper-bounds (some power of) $ F(x) - F^\star$, which is often called the Kurdyka-{\L}ojasiewicz property. This enables us to show that Assumption~\ref{EB} implies that $ \|x-\mathsf{T}_\gamma(x) \| \geq B  ~ \dist^{1-\frac{1}{p}} ~(x,\argmin F)$ for some $B>0$ (see Lemma~\ref{lem:boundKL} in  Appendix~\ref{apx:KL}) which will be used in the following results to control the iterates deviation. We refer the reader to Appendix~\ref{apx:KL} for more discussions on Error Bounds, the Kurdyka-{\L}ojasiewicz property, and their joint use.
}

\subsection{Identification Promoting Tests} \label{subsec:tests}

Let us define a set of candidate manifolds $\C = \{\M_1, \ldots, \M_N\}$ onto which identification is of particular importance. We now propose two tests that will promote identification onto these particular subspaces. The tests will be based on whether or not some points belong to a manifold of interest $\M$.

From a numerical perspective, testing if a point belongs to a set (for instance, testing which coordinates are null) is a often delicate issue due to limited machine precision. However, in our context the tested points will always be the output of a proximal operator which enjoys an explicit formulation. The considered manifolds will match this formulation. Manifold testing will thus be done naturally and exactly when computing the proximity operator.

\begin{example}
    When $g = \|\cdot\|_1$ (for instance in $\ell_1$-regularized problems), the structure of interest is the sparsity pattern, that is the collection of coordinates of $x^\star$ which are null. The proximity operator $\prox_{\gamma \|\cdot\|_1}(u)$ returns a point $x$ such that, coordinate-wise, $[\prox_{\gamma \|\cdot\|_1}(u)]_i = 0$ if and only if $u_i \in [-\gamma, \gamma]$ and $x_i = u_i - \gamma ~ \mathrm{sign}~(u_i) \neq 0$ otherwise. It is thus clear that testing if  $[\prox_{\gamma \|\cdot\|_1}(u)]_i = 0$ matches the closed-form expression of the proximity operator for the $\ell_1$ norm.
\end{example}



As we are mainly interested in final identification, it seems natural to do accelerated steps as long as the iterates are far away from the solution. Therefore, we only allow to consider non-accelerated steps when iterates live in the following set:
\begin{align*}
     \mathsf{Z} = \left\{ y : \|\mathcal T_\gamma(y) - y \|^2 \le \secondreview{\zeta} \text{ and } F(\mathcal T_\gamma (y) ) \le F(x_0) \right\},
\end{align*}
\secondreview{where $\zeta$ is any positive constant (in practice, we find that $\zeta = \|\mathcal T_\gamma(x_0) - x_0 \|^2$ is a reasonable choice).}

To determine whether $y_{k-1}\in\mathsf{Z}$ or not in \eqref{eq:condalg}, one just has to check if $\|x_k - y_{k-1} \|^2 \le \secondreview{\zeta}$ and $F(x_k)\le F(x_0)$\footnote{This functional evaluation is actually only needed if the function is sharp with $p=1$ in \cref{EB}. {For instance, this is the case when $f\equiv 0$ and $g(x) = \|x\|_1$, in which case $F(x)=\|x\|_1$ verifies \cref{EB} with $\beta=1, p=1$. In these rather degenerate cases, the proximal gradient converges in a finite number of steps.}}. Thus, it depends only on previously computed iterates and \emph{not} on the outcome of the test at time $k$.

\subsubsection{Test 1: Stopping when reaching}

As noticed in Section~\ref{sec:interplay} point (1), upon identification the momentum term is in general not aligned with the identified subspace. One further accelerated step may cause the next iterate to leave the subspace, while the vanilla proximal gradient would have stayed in it. A first natural method is thus to ``reset'' the inertial term, by performing one non-accelerated step when reaching a new manifold in our candidate set $\C$:

$\idtestargs[1] = 0$ (no acceleration) if and only if $ y_{k-1} \in \mathsf{Z}$ and
\begin{align*}
    \begin{cases}
        x_{k-1}\not\in\M       \\
        x_k\in \M \\
    \end{cases} \text{ for some } \M \in \C .
\end{align*}

Intuitively, acceleration is performed by default if the iterates are too far from optimum and as long as no new structure is identified. This means that we can benefit from the exploratory behavior of acceleration. Then, accelerated iterations will be performed as long as they do not prevent identifying a new manifold. Note that this way, if finite-time identification is possible, iterations should asymptotically all be accelerated. As expected, this method has the same rate of convergence as the accelerated proximal gradient whenever identification is possible, which is proven in the following theorem. \review{The interest of this method lies in its non-asymptotic \emph{identification} behavior, not captured by theory, which should be improved compared with proximal and accelerated proximal gradient.}

\begin{theorem}
\label{th:test1}
Let Assumptions \ref{assumption1}, \ref{hyp:alpha}, and \ref{EB} hold and take $\gamma\in(0, 1/L]$ . Then, the iterates of \eqref{eq:condalg} with test $ \idtest^{1}$ verify
\begin{align*}
   F\left(x_{n+1}\right)-F^\star \le \frac{\|x_0-x^\star\|^2}{2\gamma t_n^2} + \frac{n \mathsf{R}}{2\gamma t_n^2} = \mathcal O\left(\frac{1}{n}\right)
\end{align*}
for any $x^\star\in\argmin_{x\in\mathbb{R}^n} F(x)$ and some $\mathsf{R}>0$.

Furthermore, if Problem~\eqref{eq:pb} has a unique minimizer $x^\star$ and the qualifying constraint \eqref{eq:QC} holds for all $\M\in\C$ such that $x^\star\in\M$ at $u^\star = x^\star - \gamma \nabla f(x^\star)$, then the iterates sequence $(x_k)$ converges, finite-time identification happens, and
\begin{align*}
   F\left(x_{n+1}\right)-F^\star \le \frac{\|x_0-x^\star\|^2}{2\gamma t_n^2} + \frac{K \mathsf{R}}{2\gamma t_n^2}= \mathcal O\left(\frac{1}{n^2}\right).
\end{align*}
for some finite $K>0$.
\end{theorem}

\begin{proof}
    First, since at any iteration $k$ for which the test returned 0, we have  $\| x_k - y_{k-1} \| \le \secondreview{\zeta} $ and $  F(x_k) \le F(y_0)$ as $y_{k-1}\in\mathsf{Z}$ and thus Assumption~\ref{EB} tells us that $\|x_k-x^\star\|^2$ is bounded by some constant $\mathsf{R}$ (see Lemma~\ref{lem:boundKL} in Appendix~\ref{apx:KL} for a proof). Dropping the first term of the right hand side of \eqref{eq:base} and using the linear minoration of $t_n$ from Assumption~\ref{hyp:alpha}, we get the first result.

    Then, if there is a unique minimizer, the error bound of Assumption~\ref{EB} along with the first part of the result tells us that $x_k \to x^\star$ and thus $y_k \to x^\star$ and $u^k\to x^\star - \gamma \nabla f(x^\star)$. \reviewst{Using the qualifying constraint \eqref{eq:QC} we get for all final manifolds ({i.e.} all $\M\in\C$ such that $x^\star\in\M$) that $x_k\in\M$ after some finite time by Lemma~\ref{lemma:identif}.}\review{Since the qualifying constraint \eqref{eq:QC} holds for all final manifolds ({i.e.} all $\M\in\C$ such that $x^\star\in\M$), it follows from Lemma~\ref{lemma:identif2} that $x_k\in\M$ after some finite time. Furthermore, for the other manifolds ({i.e.} all $\M\in\C$ such that $x^\star\notin\M$), the same lemma gives us that $x_k\not\in\M$ after some finite time.}
    All in all, this means that the test $\idtest^{1}$ will produce non accelerated iterates only a finite number of times $K$ which gives the second part of the result.
\end{proof}

\subsubsection{Test 2: Prospective reach}

Another method to deal with the negative effects of inertia when the additional term is misaligned with the local manifold is to compute one proximal gradient step forward to investigate which structure can be expected from the next iterate. Intuitively, if the iterate obtained after acceleration is at least as structured as the non-accelerated one, it is kept, otherwise the point obtained after the simple proximal gradient step is taken. This is done in order to counteract both issues (1) and (2) mentioned in Sec.~\ref{sec:interplay} \review{while still benefiting from the exploratory behavior of acceleration (see point (3) in Sec.~\ref{sec:interplay})}.


$\idtestargs[2] = 0$ (no acceleration) if and only if $ y_{k-1} \in \mathsf{Z} $ and
\begin{align*}
    \begin{cases}
        \mathcal T_\gamma (x_{k}) \in\M       \\
        \mathcal T_\gamma (x_{k}+\alpha_{k}(x_{k}-x_{k-1})) \not\in\M \\
    \end{cases} \text{ for some } \M \in \C .
\end{align*}

This approach is further motivated by the desirable retraction property of the proximal (gradient) operator (see e.g. \cite[Th.~28]{DanidilisMalick2006geometryprox}).  However, a drawback of this test is the necessity to compute two proximal gradient steps (for the accelerated and the non-accelerated point) much like monotone versions of FISTA (MFISTA \cite{beck2009fast} and Monotone APG \cite{li2015accelerated}).

Similarly to \eqref{eq:condalg} with test $\idtest^1$, we expect $\idtest^2$ to provide at least a convergence similar to that of proximal gradient, and in cases when identification is possible, equivalent to that of accelerated proximal gradient.

\begin{theorem}
    \label{th:test2}
    Let Assumptions \ref{assumption1}, \ref{hyp:alpha}, and \ref{EB} hold and take $\gamma\in(0, 1/L]$. Then, the iterates of \eqref{eq:condalg} with test $ \idtest^{2}$ verify
    \begin{align*}
       F\left(x_{n+1}\right)-F^\star \le \frac{\|x_0-x^\star\|^2}{2\gamma t_n^2} + \frac{n \mathsf{R}}{2\gamma t_n^2} = \mathcal O\left(\frac{1}{n}\right)
    \end{align*}
    for any $x^\star\in\argmin_{x\in\mathbb{R}^n} F(x)$ and some $\mathsf{R}>0$.

    Furthermore, if Problem~\eqref{eq:pb} has a unique minimizer $y^\star$ and the qualifying constraint \eqref{eq:QC} holds for all $\M\in\C$ such that $x^\star\in\M$ at $u^\star = x^\star - \gamma \nabla f(x^\star)$, then the iterates sequence $(x_k)$ converges, finite-time identification happens, and
    \begin{align*}
       F\left(x_{n+1}\right)-F^\star \le \frac{\|x_0-x^\star\|^2}{2\gamma t_n^2} + \frac{K \mathsf{R}}{2\gamma t_n^2}= \mathcal O\left(\frac{1}{n^2}\right)
    \end{align*}
    for some finite $K>0$.
\end{theorem}

\begin{proof}
    The proof is the same as the one of \cref{th:test1}. Indeed, like $\idtest^1$, $\idtest^2$ is such that i) the test can return $0$ only for bounded iterates; and ii) as soon as identification happens, the test returns $1$ (i.e. acceleration).
\end{proof}

\subsection{Linear convergence rate under local restricted strong convexity} \label{subsec:lincv}

It has been observed in the literature \cite{liang2015activityDRADMM,pmlr-v80-poon18a,liang2017localcvFBtype} that algorithms showing finite time identification generally benefit from a local linear convergence property under some additional assumptions. It is also the case for \eqref{eq:condalg} with both tests.

\begin{corollary}[Linear convergence] \label{lem:lincv}
Let Assumptions \ref{assumption1}, \ref{hyp:alpha}, and \ref{EB} hold, take $\gamma\in(0, 1/L]$ and $\M$ a manifold.  Assume in addition that Problem~\eqref{eq:pb} has a unique minimizer $x^\star$, where $g$ is partly smooth relative to $\M$ and the non-degeneracy condition $-\nabla f(x^\star) \in \ri g(x^\star)$ holds. Then, \eqref{eq:condalg} equipped with either test $\idtest^{1}$ or $\idtest^{2}$ generates a sequence $(x_k)$ such that
\begin{itemize}
    \item after some finite time $K_1$, identification happens: $ \forall k\ge K_1, \quad x_k\in\M,$
    \item after some finite time $K_2>K_1$, $R$-linear convergence happens: there exists some $\alpha\in(0,1)$ and $C>0$ such that
    $$
    \forall k\ge K_2, \quad \|x_{k}-x^\star\| \le  C \alpha^{k} \|x_{K_2}-x^\star\| .
    $$
\end{itemize}
\end{corollary}

\begin{proof}
See \cref{sec:prooflincv}.
\end{proof}

\section{Numerical experiments} \label{sec:numexps}

In this section, we first show how the proposed methods can overcome the issues presented on test cases in \cref{sec:interplay}. Then, we illustrate the improved identification properties of these methods on typical machine learning/signal processing objectives.
The code used for the experiments is written in Julia \cite{bezanson2017julia} and is available at {\small \url{https://github.com/GillesBareilles/Acceleration-Identification}}.

\subsection{Application to the initial test cases}

We now return on the test cases presented in \cref{sec:interplay}, and show the iterates trajectories and suboptimality evolution along with the time of identification. For a fair comparison between \eqref{eq:condalg} with test $\idtest^2$ and the other algorithms, we plot the suboptimality versus the number of proximal gradient steps (equal to the number of iterations for all algorithms except with test $\idtest^2$ which performs two proximal gradient steps per iteration). The moment of identification of the final structure is denoted by the symbol $\oplus$ on the suboptimality plots.

In \cref{fig:numexps_lasso_full}, the $\ell_1$ norm is taken as a nonsmooth function and the candidate (linear) manifolds are the cartesian axes. Both tests allow to identify in finite time, and prevent issue (1) of Section~\ref{sec:interplay}.

In \cref{fig:numexps_1.3ball_full} and \cref{fig:numexps_2.6ball_full}, the candidate (curved) manifolds are respectively the $1.3$-norm and $2.6$-norm unit sphere. Test $\idtest^2$ allows to get finite identification, while $\idtest^1$ and accelerated proximal gradient struggle in doing so. Furthermore, the algorithm based on test $\idtest^2$ identifies the manifold as soon as one of its iterates belong to it, as opposed to the accelerated proximal gradient or $\idtest^1$.

All in all, we observe that test $\idtest^2$ corrects the problems noted in \cref{sec:interplay} on these three examples. We advocate the use of $\idtest^2$ when identifying the structure is most important. If reaching a high precision solution is the primary objective, we recommend to use test $\idtest^1$, for which each iteration is as costly as an accelerated proximal gradient one.

\begin{figure}
    \centering
    \begin{subfigure}[t]{0.49\textwidth}
        \centering
        \includegraphics[width=\textwidth]{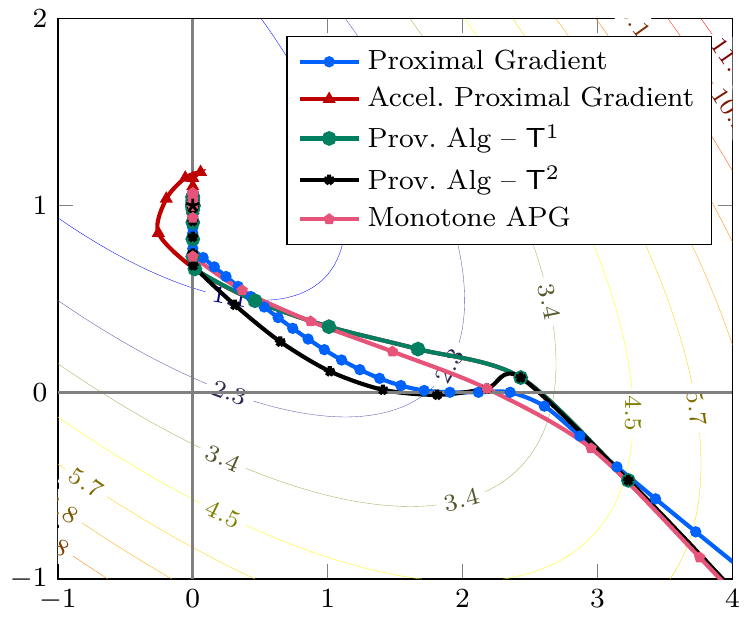}
        \caption{Iterates behavior\label{fig:numexps_LASSO_iter}}
    \end{subfigure}\hfill
    \begin{subfigure}[t]{0.49\textwidth}
        \centering
        \includegraphics[width=\textwidth]{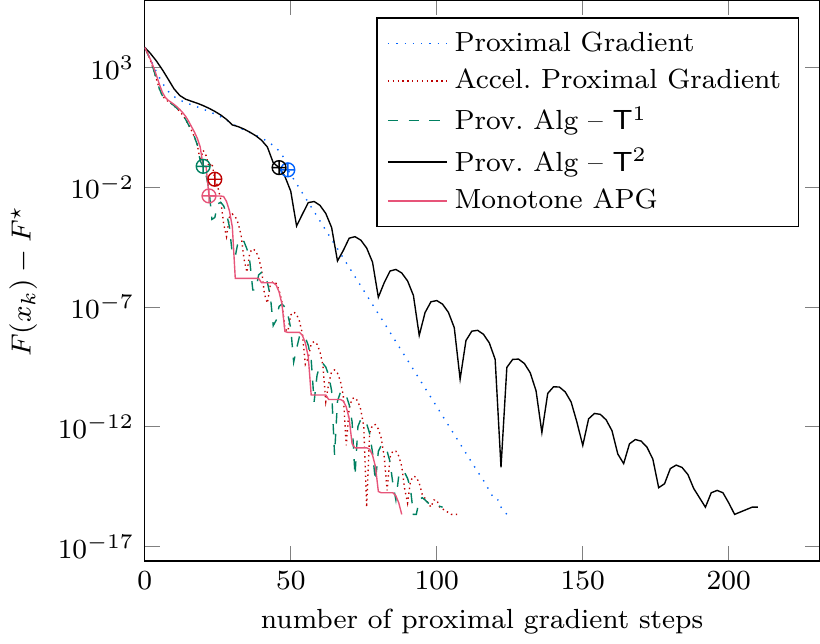}
        \caption{Suboptimality \label{fig:numexps_LASSO_subopt}}
    \end{subfigure}

    \caption{$F(x) = \|Ax-b\|^2 + g(x)$  with $g(x) = \|x\|_1$; $\M_1 = y$-axis and $\M_2 = x$-axis \label{fig:numexps_lasso_full}}
\end{figure}

\begin{figure}
    \centering
    \begin{subfigure}[t]{0.49\textwidth}
        \centering
        \includegraphics[width=\textwidth]{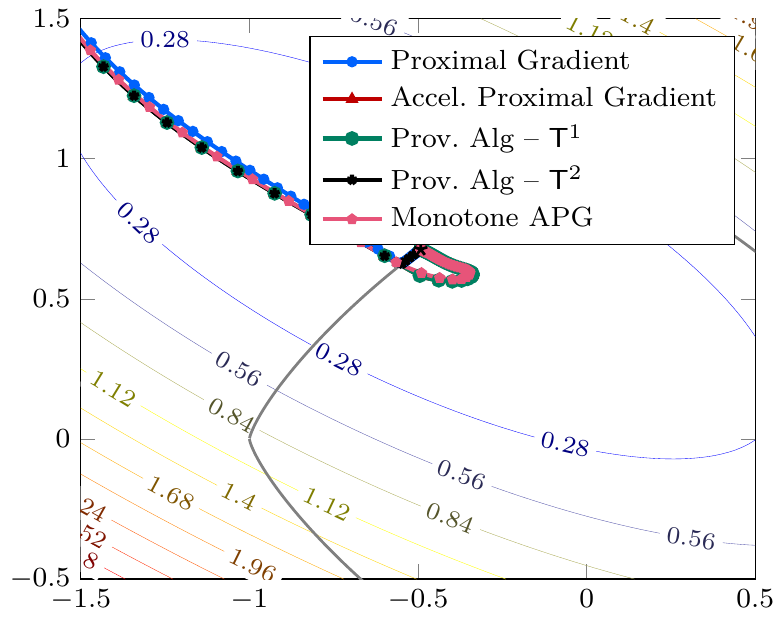}
         \caption{Iterates position\label{fig:numexps_1.3ball_iter}}
     \end{subfigure}
     \hfill
     \begin{subfigure}[t]{0.49\textwidth}
        \centering
        \includegraphics[width=\textwidth]{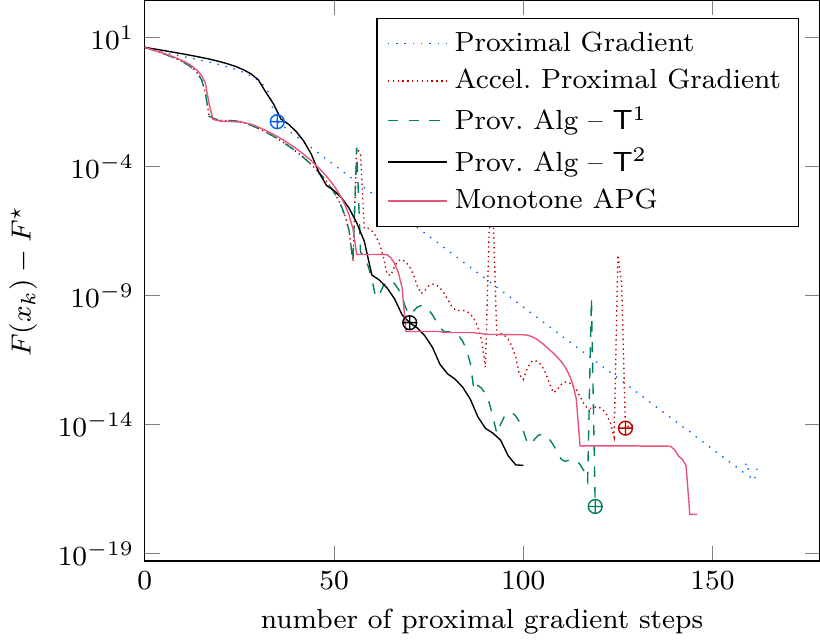}
        \caption{Suboptimality \label{fig:numexps_1.3ball_subopt}}
     \end{subfigure}

    \caption{$F(x) = \|Ax-b\|^2 + g(x)$  with  $g(x) =  \max(0, \|x\|_{1.3}-1)$; $\M = \mathcal S_{\|\cdot\|_{1.3}}(0, 1)$ \label{fig:numexps_1.3ball_full}}
\end{figure}

\begin{figure}
     \centering
     \begin{subfigure}[t]{0.49\textwidth}
         \centering
        \includegraphics[width=\textwidth]{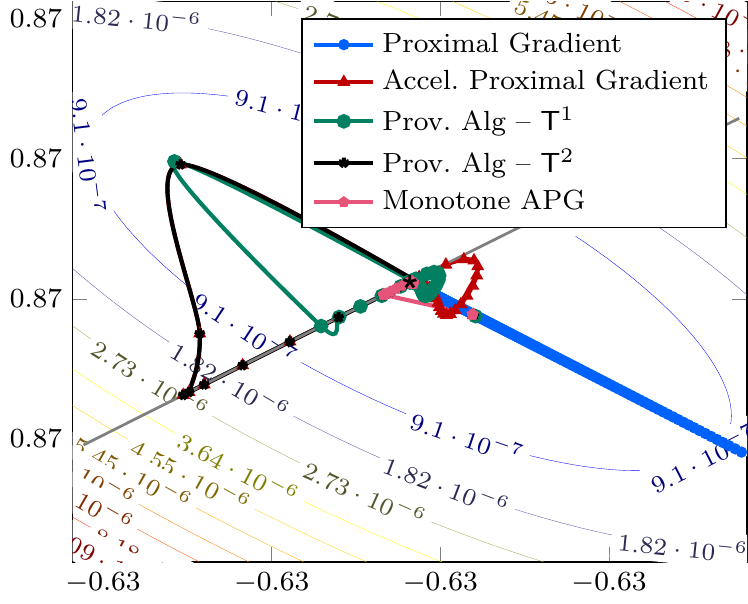}
         \caption{Iterates position\label{fig:numexps_2.6ball_iter}}
     \end{subfigure}\hfill
     \begin{subfigure}[t]{0.49\textwidth}
        \centering
        \includegraphics[width=\textwidth]{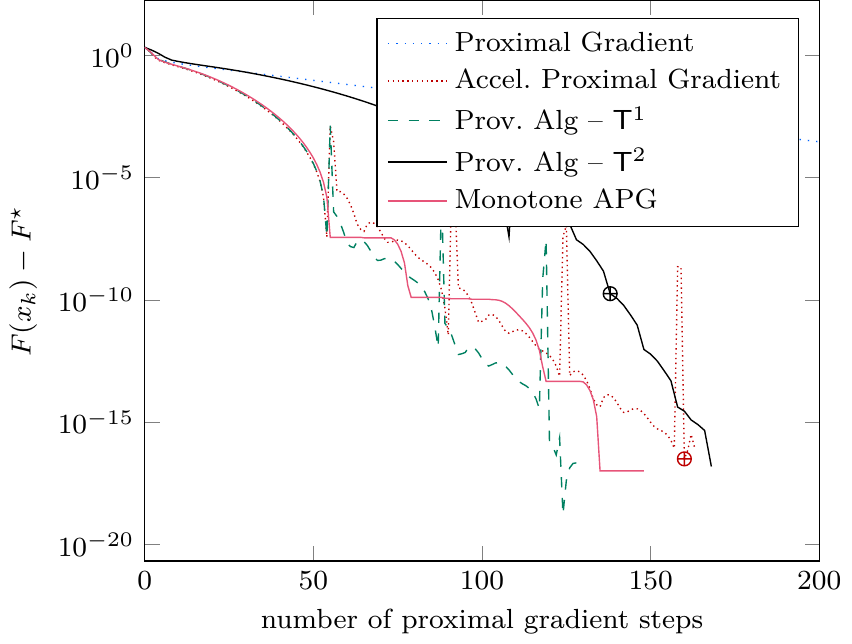}
        \caption{Suboptimality \label{fig:numexps_2.6ball_subopt}}
     \end{subfigure}
    \caption{$F(x) = \|Ax-b\|^2 + g(x)$  with  $g(x) =  \max(0, \|x\|_{2.6}-1)$; $\M = \mathcal S_{\|\cdot\|_{2.6}}(0, 1)$ \label{fig:numexps_2.6ball_full}}
\end{figure}

\subsection{Behavior for other nonsmooth structures}\label{sec:numexps_batchtests}

In this section, we illustrate the behavior of the algorithms in terms of identification. In order to do so, we generate and solve several instances of (non-strongly convex) least squares regularized problems with all algorithms. We then display \emph{per iteration} the number of final manifolds identified by the current iterate.

We consider composite problems of the form
\begin{equation*}
    \min_{x} \|Ax-b\|_2^2 + \lambda g(x)
\end{equation*}
where
\begin{itemize}
    \item $A\in\mathbb R^{m\times n}$ is a random matrix whose coefficients follow a centered reduced normal distribution;
    \item $b = As+e$ where $s$ is a structured random vector, and $e\in\mathbb R^m$ is an error term. $e$ follows a centered \reviewst{normal}\review{gaussian} distribution of variance \reviewst{$\sigma^2$ with $\sigma = 0.01$}\review{$\delta^2$ with $\delta = 0.01$};
    \item $\lambda = \delta$ so that the composite problem has the same structure as $s$ (see \cite[\reviewst{Th. 4}\review{Th. 1}]{2017-vaiter-ps-consistency}).
\end{itemize}

The considered regularizers, along with corresponding problems parameters and structure of initial signal $s$ are detailed in \cref{array:numexpspbdetail}. \reviewst{For all algorithms, the initial point $x_0$ is selected randomly.}\secondreview{For the first problem the initial point $x_0$ is drawn uniformly in $[0, 10]^n$, while for the second and third problems the entries of the initial matrix are drawn independently following the normal distribution.}

\begin{table}[h]
    \begin{tabular*}{\textwidth}{@{\extracolsep{\fill}} |cccc|}
        \hline
        Figure & nonsmooth term $g$ & Problem size & Structure to recover\\ \hline
        \cref{fig:reg_l1_randinit} & $\ell_1$ norm/lasso & $n=128$, $m=60$ & 120 null entries \\
        \cref{fig:reg_lnuclear_randinit_mean} & \secondreview{nuclear norm} & \secondreview{$n=6\times 7$, $m=2\times 2$} & \secondreview{rank 3} \\
        \cref{fig:reg_lnuclear_randinit} & nuclear norm & $n=20\times 20$, $m=16\times 16$ & rank 3 \\
        \hline
    \end{tabular*}
    \caption{Figure references and corresponding problem parameters.\label{array:numexpspbdetail}}
\end{table}


Globally, we observe that the moment of identification happens roughly at the same time for the accelerated and proposed algorithms while the vanilla proximal gradient takes more time. This justifies the use of acceleration in the first steps in order to explore correctly the search space.



Between accelerated proximal gradient and the proposed methods, \secondreview{the conclusions vary depending on whether the identified manifolds are curved or flat. \Cref{fig:reg_l1_randinit} illustrates the flat manifolds case, in which} identification happens \secondreview{soonest} with accelerated proximal gradient\secondreview{. However}, this method often looses the identified manifolds some time before complete identification (see near iteration $4.10^3$), while the proposed methods appear to identify in a more monotonous way.

\secondreview{
    The case of curved manifolds is illustrated in \cref{fig:reg_lnuclear_randinit_mean}, which shows the average percentage of identified manifolds over 30 least-square problems with nuclear norm regularization. While the proximal gradient converges too slowly to show any identification in the considered iterations range, all the other methods eventually identify the correct structure. The accelerated proximal gradient does so less efficiently in average than both proposed methods, and test $\idtest^2$ performs best both in terms of identification time and stability.
}

\secondreview{\Cref{fig:reg_lnuclear_randinit} shows a larger instance of nuclear-norm regularized least squares, in which} this stability appears even more. Indeed, while the number of correctly identified manifolds increases almost monotonically for $\idtest^2$, accelerated proximal gradient and $\idtest^1$ seem to lose all structure upon identifying a new manifold. This means that if one stops all algorithms at a $10^{-3}$ suboptimality, almost no structure is recovered for the accelerated proximal gradient, while test $\idtest^1$ and even more $\idtest^2$ are able to recover half the structure of the original signal.

\secondreview{
    All in all, the identification performance among the algorithms depends largely on whether the manifold is flat or has curvature. This can be partially explained recalling the three observations drawn on the interplay between acceleration an identification (see Section \ref{sec:interplay}). Only two apply for flat manifolds: the negative interplay of inertia misalignment -- point (1), and the positive interplay induced by the exploratory behavior of acceleration -- point (3). However, when dealing with curved manifolds the linear extrapolation of acceleration inertia naturally yields a point that does not lie on the manifold anymore -- point (2). This additional interplay explains to some extent the instability of accelerated proximal gradient noticed for curved manifolds in \cref{fig:reg_lnuclear_randinit}, and gives an intuition as to why $\idtest^2$ performs better in that case.
}

\begin{figure}
    \centering
    \begin{subfigure}[t]{0.49\textwidth}
        \centering
        \includegraphics[width=\textwidth]{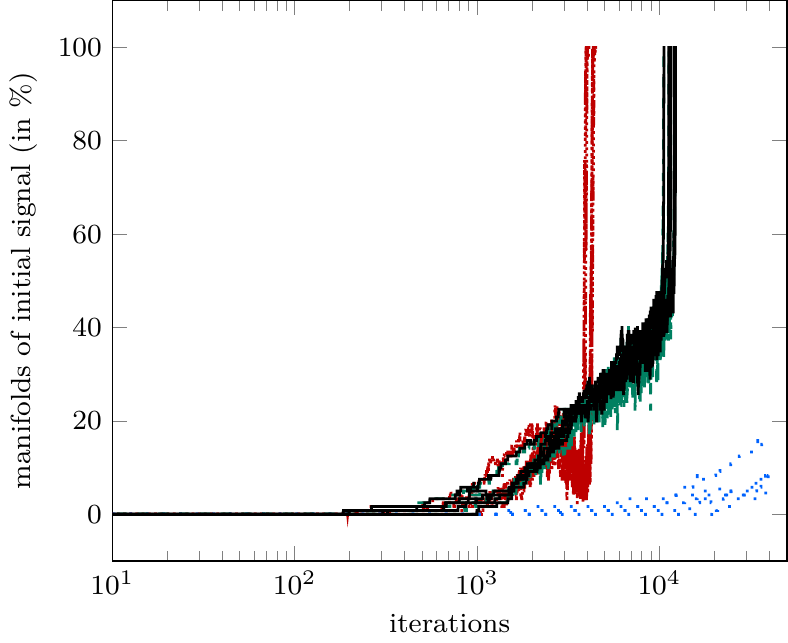}
        \caption{$g=\|\cdot\|_1$ -- 5 runs w/ random initialization -- 120 manifolds to identify\label{fig:reg_l1_randinit}}
    \end{subfigure}\hfill
    \begin{subfigure}[t]{0.49\textwidth}
        \centering
        \includegraphics[width=\textwidth]{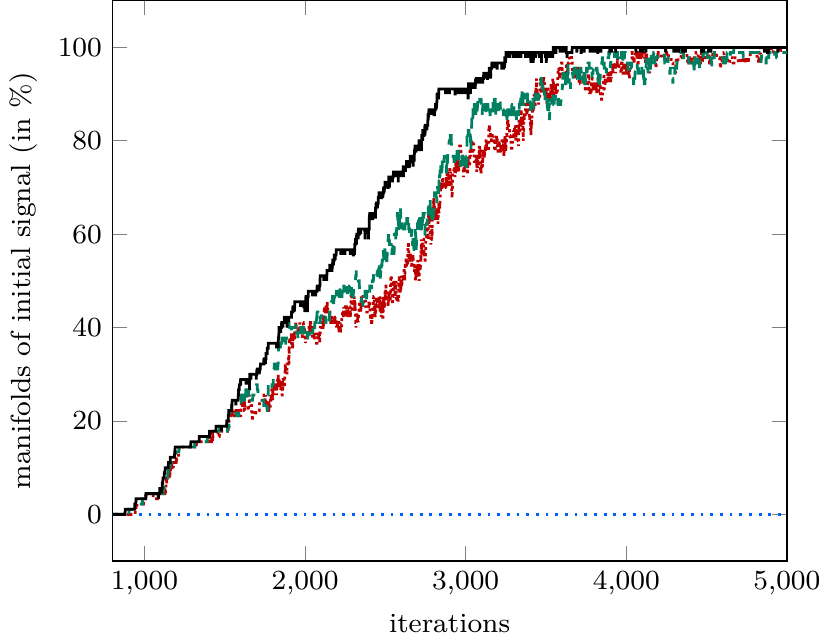}
        \caption{\secondreview{$g=\|\cdot\|_{*}$ -- average over 30 runs w/ random initialization -- 3 manifolds to be identify}\label{fig:reg_lnuclear_randinit_mean}}
    \end{subfigure} \\
    \begin{subfigure}[t]{0.49\textwidth}
        \centering
        \vspace{0.2cm}
        \includegraphics[width=0.7\textwidth]{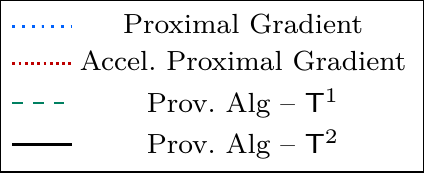}
    \end{subfigure}
    \caption{\secondreview{Identified subspaces versus number of iterations on linear (left) and curved (right) manifolds.}\label{fig:reg_l1_l12}}
\end{figure}

\begin{figure}
    \centering
    \begin{subfigure}[t]{0.49\textwidth}
        \centering
        \includegraphics[width=\textwidth]{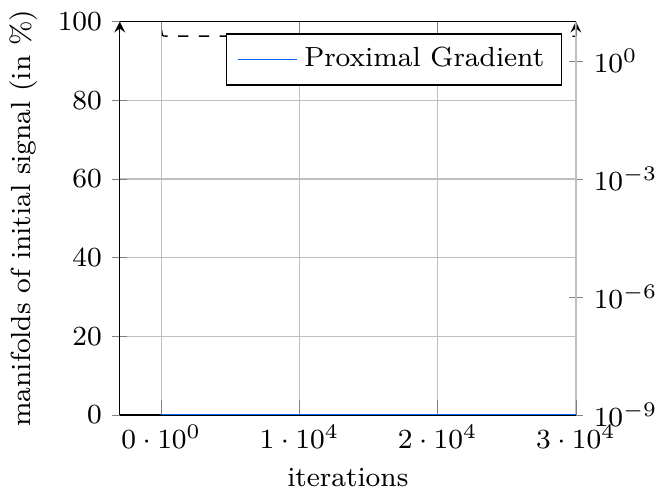}
    \end{subfigure}\hfill
    \begin{subfigure}[t]{0.49\textwidth}
        \centering
        \includegraphics[width=\textwidth]{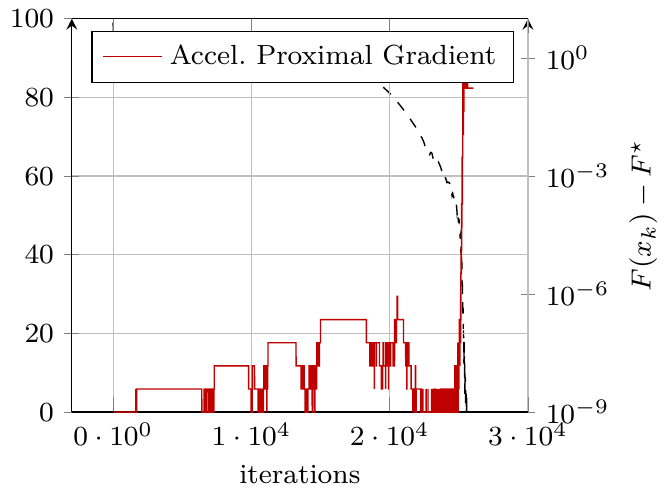}
    \end{subfigure}\\

    \begin{subfigure}[t]{0.49\textwidth}
        \centering
        \includegraphics[width=\textwidth]{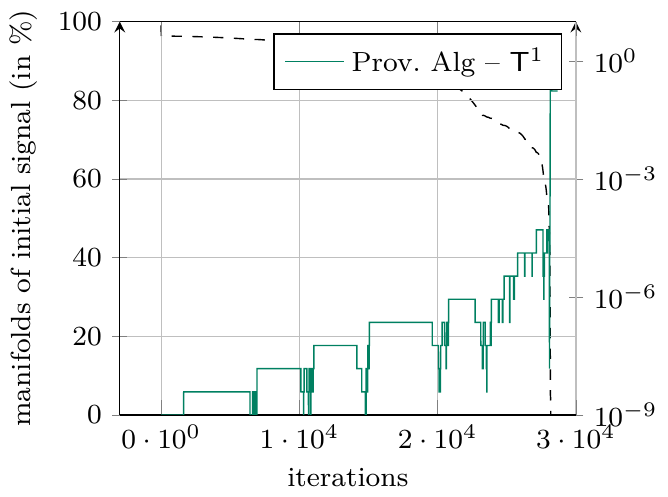}
    \end{subfigure}\hfill
    \begin{subfigure}[t]{0.49\textwidth}
        \centering
        \includegraphics[width=\textwidth]{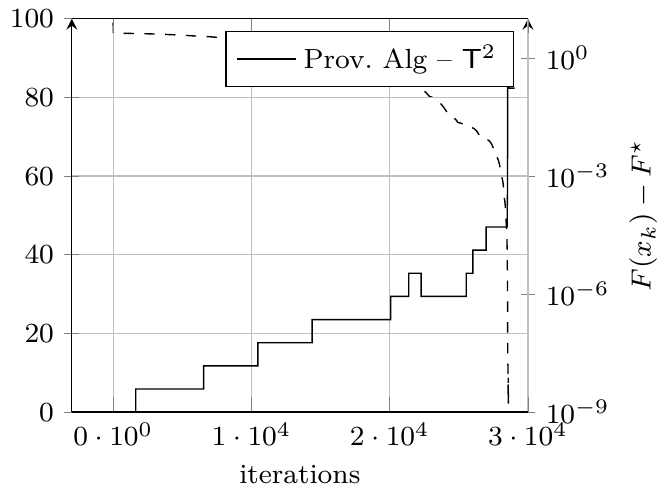}
    \end{subfigure}

    \caption{$g=\|\cdot\|_*$ -- 1 run w/ random initialization -- 17 manifolds to identify\label{fig:reg_lnuclear_randinit}}
\end{figure}


\section{Conclusion}

\review{We established that acceleration can interfere with identification for the proximal gradient algorithm, sometimes positively, sometimes negatively. However, we showed that it is possible to counteract the negative effects of acceleration on the iterates structure by not accelerating certain iterations. From this observation, we proposed two methods exhibiting a stable identification behavior while maintaining an accelerated convergence rate both in theory and in practice.}
\review{As for future directions, we mention the possiblity of extending our reasonings to more general primal-dual methods (such as ADMM, Chambolle-Pock, etc.); this would be particularly interesting in the case where two different proximity operators are involved leading to two different identification tracks.  }

\review{
\section*{Acknowledgments}

The authors would like to thank the AE and the two anonymous reviewers for their indications about the tone of the article and their suggestions for future works.
FI benefited from the support of the ANR JCJC project \emph{STROLL} (ANR-19-CE23-0008).
}

\ifdraft
    \bibliographystyle{ormsv080}
\else
    \bibliographystyle{spmpsci}
\fi
\bibliography{biblio}

\appendix

\section{Obtaining the qualifying condition} \label{appendix:QC}

The general qualifying condition \eqref{eq:QC} can be recovered under the setting of partial smoothness of $g$ and a usual non-degeneracy condition at the optimal point $x^\star$ and associated point $u^\star$ such that $x^\star = \prox_{\gamma g}(x^\star)$.

\begin{definition}[Relative interior]
The relative interior of a subset $S$, noted $\ri S$ is defined as
$$
\ri S = \{x\in S : \exists \varepsilon > 0 , \mathcal B(x, \varepsilon)\cap \Aff S \subseteq S \}
$$
where $\Aff$ denotes the affine hull of a set.
\end{definition}

\begin{definition}[Parallel space]
For a connex subset $A$, the parallel space of $A$, noted $\mathrm{par}~A$, is defined as the vector space parallel to the affine space generated by $A$.
\end{definition}

\begin{definition}[Manifold]
A subset $\M$ of $\RR^n$ is said to be a $p$-dimensional $\mathcal C^k$-submanifold of $\RR^n$ around $x \in \M$ ($1 \le k \le +\infty$) if there exists a local parameterization of $\M$ around $x$, that is, a $\mathcal C^k$ function $\varphi : \RR^p \to \RR^n$ such that $\varphi$ realizes a local homeomorphism between a neighborhood of $0 \in \RR^p$ and a neighborhood of $x \in\M$ and the derivative of $\varphi$ at $\varphi^{-1}(x)=0$ is injective.

A $p$-dimensional $\mathcal C^k$-submanifold of $\RR^n$ can alternatively be defined via a local equation, that is, a $\mathcal C^k$ function $\Phi : \RR^n \to \RR^{n-p}$ with a surjective derivative at $\bar x \in \M$, that satisfies for all $x$ close enough to $\bar x$:
$x \in\M \Leftrightarrow \Phi(x) = 0$.
\end{definition}

To lighten notations, we will only indicate the dimensionality and the smoothness degree of a manifold when it is relevant to the discussion.

\begin{definition}[Tangent space]
Given a point $x$ living in a manifold $\M$, defined either by a local parametrization $\varphi$ or a local equation $\Phi$, the tangent space of $\M$ at $x$, denoted $T_\M(x)$, is defined as:
$$
T_\M(x) = \text{Im}~ D_\varphi (0) = \text{Ker} D_\Phi (0)
$$
\end{definition}

\begin{definition}[Partly smooth function]
Let $g:\mathbb R^n\to\mathbb R$ be a proper, convex, lower semi-continuous function. $g$ is said to be partly smooth at $x$ relative to a
manifold $\M$ containing $x$ if $\partial g(x)\neq \emptyset$, and moreover
\begin{itemize}
    \item Smoothness: $\M$ is a $\mathcal C^2$-manifold around $x$, $g$ restricted to $\M$ is $\mathcal C^2$ around $x$;
    \item Sharpness: The tangent space $T_\M(x)$ coincides with $\mathrm{par} (\partial g(x))^\perp$;
    \item Continuity: The set-valued mapping $\partial g$ is continuous at $x$ relative to $\M$.
\end{itemize}
\end{definition}

\begin{lemma} \label{th:PSFimpliesQC}
Let $\M$ be a manifold and $u^\star$ a point such that $x^\star = \prox_{\gamma g}(u^\star)\in\M$. If $g$ is partly smooth at $x^\star$ relative to $\M$ and $\frac{u^\star-x^\star}{\gamma} \in \ri \partial g(x^\star)$, then \eqref{eq:QC} holds.
\end{lemma}

\begin{proof}
The proof follows the reasoning of \cite[Lem.~27, Th.~28]{DanidilisMalick2006geometryprox}. Let us define the function $\rho (u, x) = g(x) + \frac{1}{2\gamma}\|x-u\|^2$. Since $g$ is partly smooth at $x^\star$ relative to $\M$, the function $(u, x) \mapsto g(x)$ is also partly smooth at $(u^\star, x^\star)$ relative to $\RR^n\times\M$  \cite[Prop.~4.5]{lewis2002active}. Since the addition of a smooth function does not change this property  \cite[Cor.~4.6]{lewis2002active}, $\rho$ is partly smooth at $(u^\star, x^\star)$ relative to $\RR^n\times\M$.

Besides, defining the parametrized function $\rho_{u^\star}(\cdot) = \rho(u^\star, \cdot)$, we have that
\begin{itemize}
    \item[i)] $     x^\star = \prox_{\gamma g}(u^\star) = \argmin_{x\in\mathbb{R}^n} \{ g(x) + \frac{1}{2\gamma} \| x  - u^\star \|^2 \} = \argmin_{x\in\mathbb{R}^n} \rho_{u^\star} ( x ) $ and as $\rho_{u^\star}$ is $1/\gamma$-strongly convex, we have for all $x\in\mathbb{R}^n$, $  \rho_{u^\star}(x) \ge \rho_{u^\star}(x^\star) + \frac{1}{2\gamma}\|x-x^\star\|^2$;
    \item[ii)] the qualifying constraint $\frac{u^\star-x^\star}{\gamma} \in \ri \partial g(x^\star)$ implies that $0\in\ri \left( \partial g(x^\star) + \frac{x^\star-u^\star}{\gamma} \right) = \ri \partial \rho_{u^\star}(x^\star)$;
\end{itemize}
which enable us to use \cite[Lem.~27]{DanidilisMalick2006geometryprox} (or equivalently \cite[Th.~3.2 and 6.2i]{hare2004identifying}).

Thus, there exists a neighborhood $\mathcal B(u^\star, \varepsilon)$ of $u^\star$ and a function $\Phi$ such that for all $u$ in $\mathcal B(u^\star, \varepsilon)$, $\Phi(u)\in\M$ and is a critical point of $\rho_u$ restricted to a neighborhood of $x^\star$. Since $\rho_u$ is convex, $\Phi(u)$ is actually the global optimum of $\rho_u$. Therefore, we have that  for all $u$ in $\mathcal B(u^\star, \varepsilon)$, $\prox_{\gamma g}(u)\in \M$ which is exactly \eqref{eq:QC}.
\end{proof}

\section{Recalls on the (Accelerated) Proximal Gradient descent} \label{appendix:functionalcv}

We recall here the descent lemma for the composite objective function $F$, which is central in the analysis of any first order method.

\begin{lemma}[Descent lemma]\label{lem:descent} Let $\gamma>0$, the following inequalities hold for any $x, y \in \mathbb R^n$
\begin{align*}
F\left(\T(x)\right) + \frac{(1-\gamma L)}{2 \gamma}\left\|\T(x)-x\right\|^{2}+\frac{1}{2 \gamma}\left\|\T(x)-y\right\|^{2} &\le F(y) + \frac{1}{2 \gamma}\|x-y\|^{2}\\
F\left(\T(x)\right) + \frac{(2-\gamma L)}{2 \gamma}\left\|\T(x)-x\right\|^{2} + \frac{1}{ \gamma} \langle x-y, \T(x)-x \rangle &\le F(y)
\end{align*}
\end{lemma}

\begin{proof}
The first inequality is directly equivalent to \cite[Lem.~1]{chambolle2015convergence}. The second is derived from the identity $ \|\T(x)-x\|^2 + \|y-x\|^2 = 2\langle \T(x)-x, y-x \rangle + \|y-\T(x)\|^2 $.
\end{proof}

We now give an accelerated descent lemma for the composite objective $F$, that is fundamental in the analysis of our provisional algorithm in \cref{subsec:condalg}.

\begin{lemma}[Accelerated Descent lemma]\label{lem:accdescent}
    Let Assumptions \ref{assumption1} and \ref{hyp:alpha} hold. For any $\gamma>0$, any pair of points $(x_k, y_k)$ and for $x^\star$ a fixed point of $\mathcal T_{\gamma}$ (i.e. a minimizer of $F$), we have (independently of the acceleration step)
    \begin{align*}
        t_{k}^2 v_k - t_{k-1}^2 v_{k-1} \le & -\frac{1-\gamma L}{2 \gamma} \|t_k x_{k+1}- t_k y_k\|^2  \\
      \nonumber  & ~~~ - \frac{1}{2\gamma}\|t_k x_{k+1} - (t_k-1)x_k -x^\star\|^2 + \frac{1}{2\gamma} \|t_k y_k - (t_k-1)x_k -x^\star\|^2 ,
    \end{align*}
    where $x_{k+1} = \mathcal T_{\gamma}(y_k)$, and $v_k:=F(x_{k+1})-F^\star$.
\end{lemma}

\begin{proof}
    The proof follows the same global layout as the one of FISTA \cite{beck2009fista}.  Using the Lemma~\ref{lem:descent} at $(x=y_k, y=x_k)$ and $(x=y_k, y=x^\star)$,
    \begin{align*}
    v_k-v_{k-1} &= F(x_{k+1}) - F(x_k) \le -\frac{2-\gamma L}{2 \gamma}\|x_{k+1}-y_k\|^2 - \frac{1}{ \gamma} \langle y_k-x_k, x_{k+1}-y_k \rangle \\
    v_k &= F(x_{k+1}) - F^\star \le -\frac{2-\gamma L}{2 \gamma}\|x_{k+1}-y_k\|^2 - \frac{1}{ \gamma} \langle y_k-x^\star, x_{k+1}-y_k \rangle
    \end{align*}

    The first equation times $(t_k-1)$ added to the second yields:
    \begin{align*}
    t_k v_k - (t_k-1)v_{k-1} &\le -\frac{2-\gamma L}{2 \gamma} t_k \|x_{k+1}-y_k\|^2 - \frac{1}{ \gamma} \langle t_k y_k - (t_k-1)x_k -x^\star, x_{k+1}-y_k \rangle
    \end{align*}

    Multiplying by $t_k$, using the relation $t_k^2-t_k \le t_{k-1}^2$ from \cref{hyp:alpha}(i), we get:
    \begin{align*}
    t_{k}^2 v_k - t_{k-1}^2 v_{k-1}
    &\le -\frac{2-\gamma L}{2 \gamma} \|t_k x_{k+1}- t_k y_k\|^2 - \frac{1}{ \gamma} \langle t_k y_k - (t_k-1)x_k -x^\star, t_k x_{k+1}- t_k y_k \rangle
    \end{align*}

    Applying the identity $\|c-a\|^2+2\langle a-b,c-a\rangle = \|b-c\|^2 - \|b-a\|^2$ to the inner product of the last identity yields the result.
\end{proof}

\section{About Error Bounds}
\label{apx:KL}

In order to connect the functional suboptimality with the iterates behavior, the geometry of the function can be used. For a proper convex lower-semicontinuous function $F$ achieving its minimum $F^\star$, two common variants of geometric inequalities are Error Bounds and Kurdyka-{\L}ojasiewicz inequalities \cite{Bolte2007KL,bolte2017error,thesetrong}. For any\footnote{These properties are often supposed to hold only locally but can be globalized easily \cite[Sec.~2.4, Cor.~6]{bolte2017error}.} $x \notin \argmin F$, they respectively write
\begin{itemize}
    \item \emph{Error Bound} ~~~~~~~~~~~~~~~~~~~~~~~~~~ $\varphi(F(x)-F^\star)\geq \dist(x,\argmin F)$
    \item \emph{Kurdyka-{\L}ojasiewicz} (KL) ~~~~~~~~ $ \dist(0,\partial F(x))\geq 1/\varphi'(F(x)-F^\star)$
\end{itemize}
where $\varphi(t) = Ct^\theta /\theta$ with $C>0$ and $\theta\in(0,1]$ is called a {desingularizing} function. These properties are widely satisfied (e.g. by any semi-algebraic function \cite{Bolte2007KL}) but the desingularizing function (more particularly $\theta$) is often hard to estimate\footnote{An example of such a computation is performed for the lasso problem $ F(x) = \frac{1}{2} \left\|Ax-b\right\|_2^2 +  \lambda_1 \|x\|_1$ in \cite[Lem.\;10]{bolte2017error}}.

It is easy to see that Assumption~\ref{EB} is a (global) error bound with $\varphi(t) =  (1/\beta)^{1/p} t^{1/p}$ (matching the definition of $\varphi$ with $C = 1/p (1/\beta)^{1/p}>0$ and $\theta=1/p\in(0,1]$)  but the knowledge of the constants is not necessary.

Then, an important result about Error Bounds and KL inequalities is that they are equivalent with the same desingularizing function \cite[Th.~5]{bolte2017error} which allows us to get the following result for the proximal gradient operator.

\begin{lemma}
    \label{lem:boundKL}
Let Assumptions~\ref{assumption1} and \ref{EB} hold. Then, there exists a constant $B>0$ such that for all $x$
\begin{align*}
     \|x-\mathsf{T}_\gamma(x) \| \geq
B  ~ \dist^{1-\frac{1}{p}} ~(x,\argmin F) .
\end{align*}
\end{lemma}

\begin{proof}
First, let us notice that the KL and error bound properties can be combined as
\begin{align*}
\dist(0,\partial F(x)) \geq
\frac{1}{C} ( F(x)-F^\star )^{1-\theta}  \geq
\frac{1}{C} \left(\frac{\theta}{C}\right)^{\frac{1-\theta}{\theta}} \left(\dist(x,\argmin F)\right)^{\frac{1-\theta}{\theta}}  .
\end{align*}
Furthermore, using that $
\dist(0,\partial F(\mathsf{T}_\gamma(x) )) \leq \frac{L\gamma + 1}{\gamma} \|x-\mathsf{T}_\gamma(x) \| $ for all $x\in \mathbb{R}^n$ (see e.g. \cite[Sec.~4.1]{bolte2017error}), we get that
\begin{align*}
 \|x-\mathsf{T}_\gamma(x) \| \geq
\frac{\gamma}{C(L\gamma+1)} \left(\frac{\theta}{C}\right)^{\frac{1-\theta}{\theta}} \left(\dist(x,\argmin F)\right)^{\frac{1-\theta}{\theta}}  ,
\end{align*}
which, combined with the fact that $\theta=1/p$, gets the claimed result.
\end{proof}

\section{Proof of Corollary~\ref{lem:lincv}} \label{sec:prooflincv}

\begin{proof}
First, \cref{th:test1} or \cref{th:test2} give the convergence of the sequence $(x_k)$ to the unique minimizer $x^\star$ and finite time convergence since the partial smoothness assumption of \cref{lem:lincv} implies the qualifying constraint \eqref{eq:QC} of both theorems from \cref{th:PSFimpliesQC}.

In order to show the linear convergence behavior, we follow \cite{liang2017localcvFBtype} and \cite[Chap. 2.1.2]{polyakintroduction} and apply their first-order analysis of accelerated proximal gradient to \eqref{eq:condalg}. We are in the same setting and have compatible assumptions so that we can apply \cite[Prop.~4.5]{liang2017localcvFBtype} to \eqref{eq:condalg} as often as identification happened since both tests will return acceleration after that moment. We have, for $r_k=x_k-x^\star$, $d_k = [r_{k-1} \quad r_k]^T$:
\begin{equation} \label{eq:linearizedFISTA}
    d_{k+1} = M d_k + e_k, \qquad \|e_k\| = o(\|d_k\|)
\end{equation}
with $M$ a matrix of spectral radius $\rho(M)<1$ as $\gamma \in (0, 1/L]$ from \cite[Cor.~4.9, Rem.~4.10]{liang2017localcvFBtype}.

Take any $\varepsilon \in(0,(1-\rho(M))/2)$ and let $K_2>0$ be the smallest time from which i) identification holds; and ii) $\|e_k\| / \|d_k\| \leq \varepsilon$. $K_2$ is finite since identification happens in finite time and $\|e_k\| = o(\|d_k\|)$.

Expliciting \cref{eq:linearizedFISTA} yields, for $k\geq 0$,
\begin{equation*}
    d_{K_2+k+1} = M^k d_{K_2+1} + \sum_{l=1}^k M^{k-l}e_{K_2+l}.
\end{equation*}

Taking the norm and applying triangular inequalities yield
\begin{align*}
    \|d_{K_2+k+1}\| &\le \|M^k\| \|d_{K_2+1}\| + \sum_{l=1}^k \|M^{k-l}\|\|e_{K_2+l}\| \le C_M \rho^k \|d_{K_2+1}\| + C_M \sum_{l=1}^k \rho^{k-l}\|e_{K_2+l}\|
\end{align*}
where $\rho:=\rho(M) +\varepsilon$ and $C_M>0$ is a constant such that $\|M^k\| \leq C_M \rho^k $ for all $k$ (see \cite[Cor.~5.6.13]{horn2012matrix}).

Let us denote $w_{k+1} = C_M^{-1} (\rho + \varepsilon)^{-k}\|d_{K_2+k+1}\|$  ($\rho + \varepsilon = \rho(M) + 2\varepsilon \in (0,1)$) and show that this sequence is uniformly bounded by a constant $C_w$. The initialization is clear. Suppose that $w_l\leq C$ for all $l\leq k$; multiplying the previous equation by $C_M^{-1} (\rho+\epsilon)^{-k}$ yields
\begin{align*}
    w_{k+1} &\le \left(\frac{\rho}{\rho+\varepsilon}\right)^k \|d_{K+1}\| +  \frac{1}{\rho+\epsilon}\sum_{l=1}^k \left(\frac{\rho}{\rho+\epsilon}\right)^{k-l} \frac{\|d_{K+l}\|}{(\rho+\epsilon)^{l-1}}\frac{\|e_{K+l}\|}{\|d_{K+l}\|} \\
    &=  \left(\frac{\rho}{\rho+\epsilon}\right)^k w_1 +  \frac{1}{\rho+\epsilon}\sum_{l=1}^k \left(\frac{\rho}{\rho+\epsilon}\right)^{k-l} w_l \frac{\|e_{K+l}\|}{\|d_{K+l}\|}
\end{align*}

By recursion, we have
\begin{align*}
    w_{k+1} &\le C_w \left(\frac{\rho}{\rho+\epsilon}\right)^k  + C_w \frac{\epsilon}{\rho+\epsilon}  \sum_{l=1}^k \left(\frac{\rho}{\rho+\epsilon}\right)^{k-l} \\
    & = C_w  \left(\frac{\rho}{\rho+\epsilon}\right)^k + C_w \frac{\epsilon}{\rho+\epsilon} \frac{1-\left(\frac{\rho}{\rho+\epsilon}\right)^k}{1-\frac{\rho}{\rho+\varepsilon}} = C_w
\end{align*}

Therefore, for all $k\ge 0$,  $\|d_{K_2+k+1}\| \le C_M C_w (\rho(M) + 2\varepsilon)^k \|d_{K+1}\|$ for any $\varepsilon>0$.
\end{proof}

\section{Additional Illustrations}

We provide here additional illustrations \secondreview{\cref{fig:reg_l1_zeroinit} and \cref{fig:reg_lnuclear_zeroinit}}, which consist in the same experiments as \secondreview{in \cref{fig:reg_l1_randinit}} and \cref{fig:reg_lnuclear_randinit} except for the starting point of the algorithms\secondreview{, which are taken here as} the null vector (or matrix). The conclusions are quite similar to those of  \cref{sec:numexps_batchtests}; the identification stability of the proposed methods is even more explicit.

\secondreview{
    Besides, we also report one further experiment where the solution is non qualified, and the structure-encoding manifolds have non-null curvature. The smooth part $f$ is again a least-squares function, while the non-smooth part is defined as $g(x) = \max(0, \|x_{1:5}\|_{1.3}-1) + \ldots + \max(0, \|x_{45:50}\|_{1.3}-1)$ for $x\in\mathbb R^{50}$. The structure-encoding manifolds of $g$ are any cartesian product $\M = \times_{i=1}^{10} \M_i \subset\mathbb R^{50}$, where each manifold $\M_i$ is either $\mathbb R^5$ or $\{x\in\mathbb R^5 : \|x\|_{1.3}=1\}$ (i.e. the 1.3-norm unit sphere). Fig.~\ref{fig:reg_distball_randinit} displays the performance of the four algorithms of interest (stopped as soon as suboptimality gets below $10^{-12}$). None is able to recover the complete structure of the solution, which lives in $\M^\star = \times_{i=1}^{10} \{x\in\mathbb R^5 : \|x\|_{1.3}=1\}$. This is not surprising for a non-qualified problem. The vanilla proximal gradient recovers half of the elementary manifolds, and $\idtest^2$ recovers 4 but much faster. Interestingly, the accelerated proximal gradient and $\idtest^1$ iterates are not able to retain even one of the manifolds they encountered. This behavior is consistent with the one noticed in \cref{fig:1.3ball} and suggests that proximal gradient may have more robust identification properties than accelerated proximal gradient for non-qualified problems. The second test seems to behave comparably to proximal gradient descent in terms of identification, while converging with the same number of iterations as accelerated proximal gradient.
}

\begin{figure}[hp]
        \centering
        \includegraphics[width=0.70\textwidth]{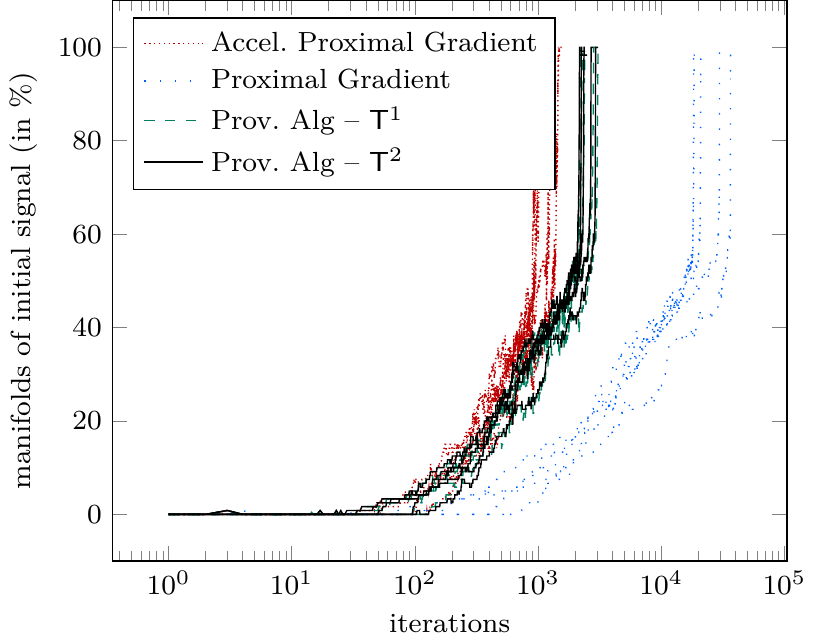}
        \caption{$g=\|\cdot\|_1$ -- 5 runs w/ initialization at $0$ -- 120 manifolds to identify\label{fig:reg_l1_zeroinit}}
\end{figure}

\begin{figure}
    \centering
    \begin{subfigure}[t]{0.49\textwidth}
        \centering
        \includegraphics[width=\textwidth]{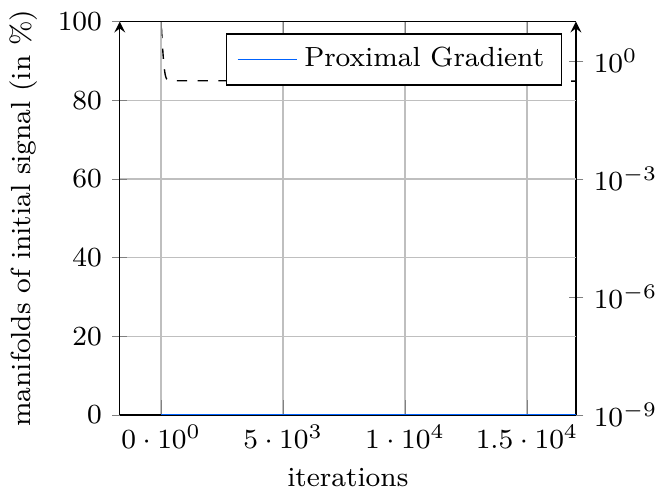}
    \end{subfigure}\hfill
    \begin{subfigure}[t]{0.49\textwidth}
        \centering
        \includegraphics[width=\textwidth]{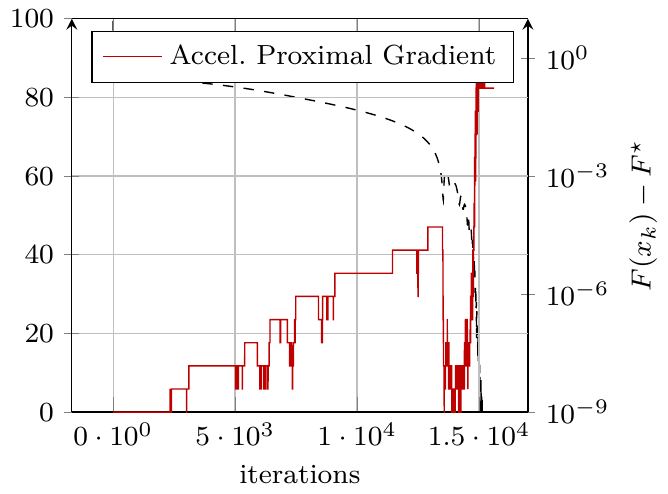}
    \end{subfigure}\\

    \begin{subfigure}[t]{0.49\textwidth}
        \centering
        \includegraphics[width=\textwidth]{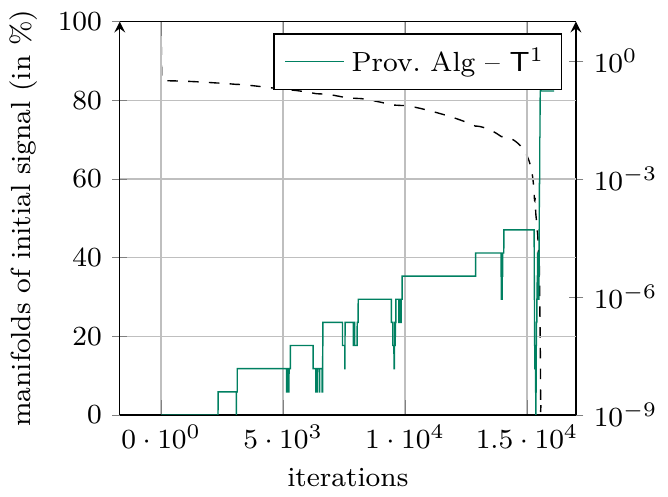}
    \end{subfigure}\hfill
    \begin{subfigure}[t]{0.49\textwidth}
        \centering
        \includegraphics[width=\textwidth]{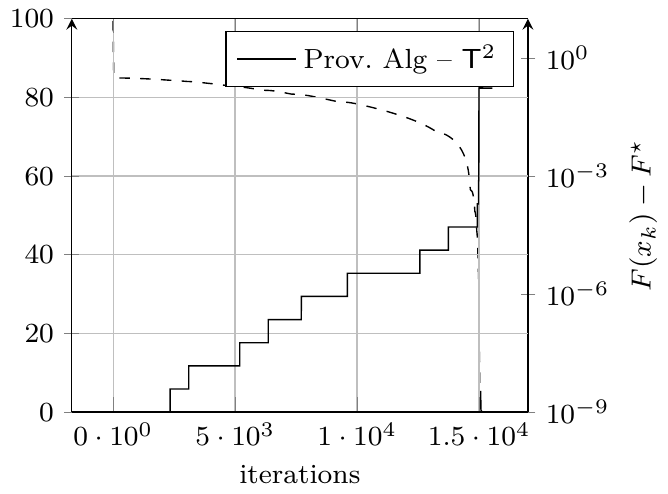}
    \end{subfigure}

    \caption{$g=\|\cdot\|_*$ -- 1 run w/ null initialization -- 17 manifolds to identify\label{fig:reg_lnuclear_zeroinit}}
\end{figure}

\begin{figure}[hp]
    \centering
    \includegraphics[width=0.60\textwidth]{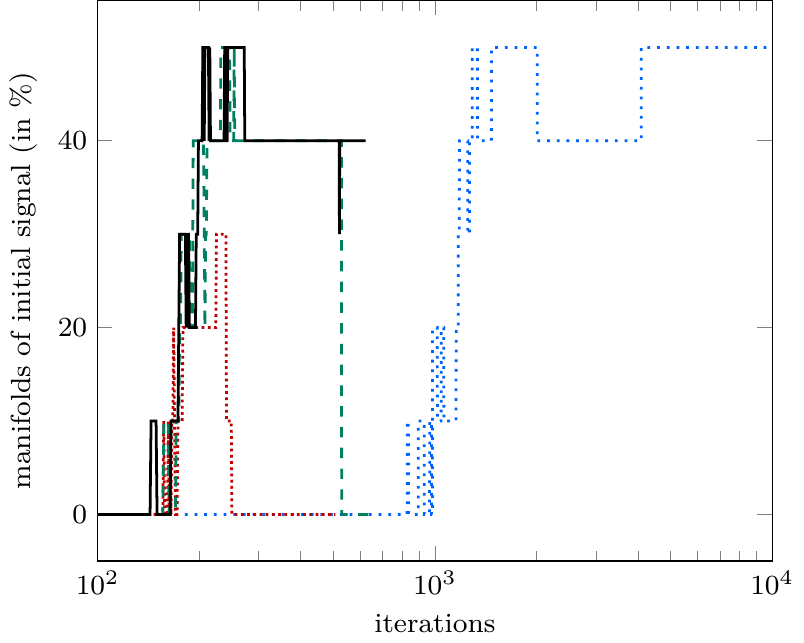}
    \caption{\secondreview{$g(x) = \max(0, \|x_{1:5}\|_{1.3}-1) + \ldots + \max(0, \|x_{45:50}\|_{1.3}-1)$ -- 1 runs w/ random initialization -- 10 manifolds to identify\label{fig:reg_distball_randinit}}}
\end{figure}

\end{document}